\documentclass[noinfoline]{imsart}
\usepackage{amsmath,amssymb,amsthm,booktabs}
\usepackage{mathrsfs}
\usepackage{amsmath,amssymb,amsthm,enumerate,framed,graphicx,array,color,multirow,mathrsfs,amsrefs}

\usepackage[utf8]{inputenc}
\usepackage[colorlinks,citecolor=blue,urlcolor=blue]{hyperref}

\usepackage{bm}
\usepackage{euscript}
\usepackage{graphicx}

\usepackage{multicol}
\usepackage[usenames,dvipsnames,svgnames,table]{xcolor}

\usepackage{a4wide}

\usepackage{mathrsfs}
\usepackage{euscript}

\startlocaldefs
\numberwithin{equation}{section} \theoremstyle{plain}
\newtheorem{theorem}{Theorem}[section]
\newtheorem{lemma}{Lemma}[section]
\newtheorem{corollary}{Corollary}[section]

\newtheorem{definition}{Definition}
\newtheorem{remark}{Remark}[section]
\endlocaldefs

\allowdisplaybreaks

\allowdisplaybreaks[4]
\numberwithin{equation}{section}

\def\bC{\mathbb{C}}
\def\bN{\mathbb{N}}
\def\bR{\mathbb{R}}
\def\bE{\mathbb{E}}
\def\bP{\mathbb{P}}
\def\bT{\mathbb{T}}

\def\cF{\mathcal{F}}

\def\cE{\mathcal{E}}

\def\cC{\mathcal{C}}
\def\cI{\mathcal{I}}

\def\spec{\bold{Spec}}

\def\dimh{\dim_{_{\rm H}}}

\def\1{\mathbf 1}
\def\im{\mathrm{Im}}

\begin{document}

\title{On collision of multiple eigenvalues for  matrix-valued Gaussian processes}

\runtitle{Collision of multiple eigenvalues}

  \begin{aug}
    \author{\fnms{~ Jian} \snm{Song}\ead[label=e1]{txjsong@hotmail.com}}
    % \footnote{Li's  research was supported by was supported by NIDA, NIH grants
    % P50 DA039838,  a NSF grant DMS 1512422 and National Nature Science Foundation of
    % China (NNSFC), 11690015.}
    \and
    \author{\fnms{~ Yimin} \snm{Xiao~}\ead[label=e2]{xiaoy@msu.edu }}
    \and
    \author{\fnms{~ Wangjun} \snm{Yuan}\ead[label=e3]{ywangjun@connect.hku.hk}}
    
    \affiliation{Shandong University and  The University of Hong Kong}
    \runauthor{J. Song,   Y. Xiao \& W. Yuan}

    \address{Research Center for Mathematics and Interdisciplinary Sciences, Shandong University, Qingdao, Shandong, 266237, China;
    and School of Mathematics, Shandong University, Jinan, Shandong, 250100, China\\
      \printead{e1}}
    
    \address{ Department of Statistics and Probability, Michigan State University, A-413 Wells Hall, East Lansing,
MI 48824, U.S.A.\\
      \printead{e2}
   }

    \address{
      Department of Mathematics, 
      The University of Hong Kong\\
      \printead{e3}
    }
  \end{aug}



\date{}

\begin{abstract}  For  real symmetric and  complex Hermitian Gaussian processes  whose values are  $d\times d$ 
matrices, we characterize the conditions under which the probability that at least $k$ eigenvalues  collide is positive for 
$2\le k\le d$, and we obtain the Hausdorff dimension of the set of collision times. 
\end{abstract}

\maketitle

{
\hypersetup{linkcolor=black}
 \tableofcontents 
}

\section{Introduction}
\subsection{Background  and the main results}

Our work is mainly motivated by the recent work of Jaramillo and Nualart  \cite{Jaramillo2018} 
who considered  the problem on the existence of collision of the eigenvalues of Gaussian orthogonal 
ensemble (GOE) process and Gaussian unitary ensemble (GUE) process associated with Gaussian random fields.

Let $N \in \bN$ be fixed and consider a centered Gaussian random field $\xi = \{\xi(t): t \in \bR_+^N \}$ defined on a probability 
space $(\Omega, \cF, \bP)$ with covariance given by
\begin{align*}
	\bE \left[ \xi(s) \xi(t) \right] = C(s,t),
\end{align*}
for some non-negative definite function $C: \bR_+^N \times \bR_+^N \rightarrow \bR$. Let $\{\xi_{i,j}, \eta_{i,j}:i,j \in \bN\}$ be a family of
 independent copies of $\xi$. For $\beta \in \{1,2\}$, and $d \in \bN$ with $d \ge 2$ fixed, consider the following $d\times d$ matrix-valued 
 process $X^{\beta} = \{X_{i,j}^{\beta}(t); t \in \bR_+^N, 1 \le i,j \le d\}$ with entries given by
\begin{align} \label{def-entries}
	X_{i,j}^{\beta}(t) =
	\begin{cases}
	\xi_{i,j}(t) + \iota \1_{[\beta = 2]} \eta_{i,j}(t), & i < j; \\
	\sqrt{2} \xi_{i,i}(t), & i=j; \\
	\xi_{j,i}(t) - \iota \1_{[\beta = 2]} \eta_{j,i}(t), & i > j.
	\end{cases}
\end{align}
Throughout this paper, we denote by $\iota = \sqrt{-1}$  the imaginary unit.  Clearly, for every $t \in  \bR_+^N$, $X^{\beta}(t)$ is a real  
symmetric matrix for $\beta=1$ and a  complex  Hermitian matrix for $\beta=2$.  In particular,  $X^1(t)/\sqrt{C(t,t)}$ belongs to  GOE 
and $X^2(t)/\sqrt{2C(t,t)}$ belongs to  GUE, respectively.

Let $A^1$ be a real symmetric deterministic matrix and $A^2$ be a complex Hermitian deterministic matrix. Suppose 
that $\{\lambda_1^{\beta}(t), \cdots,  \lambda_d^{\beta}(t)\}$ is the set of eigenvalues of 
\begin{align}\label{e:Y}
      Y^{\beta}(t) = A^{\beta} + X^{\beta}(t)
\end{align}
 for $\beta=1,2$.

For $a = (a_1, \ldots, a_N), \, b = (b_1, \ldots, b_N) \in \bR_+^N$  satisfying $a_i \le b_i$ for $1 \le i \le N$, let $I = [a,b] $ be the 
closed interval defined by 
\begin{align} \label{def-interval}
	I = [a,b] = \prod_{j=1}^N[a_j, b_j] \subseteq \bR_+^N.
\end{align}

In this paper, we consider the following question on collisions of the eigenvalues of $X^{\beta}$: for $k \in \{2,3,\dots, d\}$, when can 
\begin{equation}\label{Eq:Question}
\bP \left( \lambda_{i_1}^{\beta}(t) = \cdots = \lambda_{i_k}^{\beta}(t) \mathrm{\ for \ some \ } t \in I \ \mathrm{and} \ 
1 \le i_1 < \cdots < i_k \le d \right)> 0~?
\end{equation}
This question was first studied by Dyson in a pioneering and fundamental work \cite{Dyson1962} for $N= \beta  = 1$,
$k = 2$,  and $\xi$ being a standard Brownian motion. Dyson proved that the eigenvalue processes $\{ \lambda_i^{\beta}(t), t\ge 0\}$
($i = 1, \ldots, d$) satisfy a system of the It\^o stochastic differential equations with non-smooth diffusion coefficients 
and never collide, i.e., 
\begin{equation}\label{Eq:Dyson}
\bP \left( \lambda_{i}^{1}(t)  = \lambda_{j}^{1}(t) \mathrm{\ for \ some \ } t >0 \ \mathrm{ and} \ 
1 \le i < j \le d \right) =0.
\end{equation}
The stochastic process $\{ (\lambda_1^{\beta}(t), \ldots, \lambda_d^{\beta}(t)),\, t\ge 0\}$
 is called the Dyson non-colliding Brownian motion. For more information, see Anderson et al. \cite{anderson2010}.
Nualart and P\'erez-Abreu  \cite{nualart2014eigenvalue} 
 proved that  \eqref{Eq:Dyson} still holds in
the case where $N=\beta  = 1$ and $\xi$ is a Gaussian process with H\"older continuous paths of the order larger
than 1/2;  furthermore, it was also shown  that the eigenvalues of a symmetric random matrix associated with fractional Brownian motion with index $H \in(1/2,1)$ satisfy a system of equations, which is an extension of Dyson's SDEs to the case of fractional Brownian motion.  More recently, Jaramillo and Nualart \cite{Jaramillo2018} studied Question \eqref{Eq:Question} for $k=2$ 
and $X^{\beta}$ ($\beta = 1, 2$) that are associated with a general class of Gaussian random fields. They provided 
optimal sufficient condition in terms of the Bessel-Riesz capacity and a necessary condition in 
terms of Hausdorff measure for \eqref{Eq:Question} to hold with $k= 2$.

 We also would like to mention some progress on the study of the eigenvalues  of matrix-valued processes driven by fractional Brownian motion. Pardo et al. \cite{Pardo2016} obtained high-dimensional convergence in distribution for the empirical spectral measure of a scaled symmetric fractional Brownian matrix, and the result  was extended to centered Gaussian processes in Jaramillo et al.  \cite{Jaramillo2019}. For a scaled fractional Wishart matrix, Pardo et al. \cite{Pardo2017} obtained SDEs for the eigenvalues, conditions for non-collision of eigenvalues,  and the high-dimensional convergence in distribution of the empirical spectral measure. Recently, Song et al. \cite{SYY20} obtained high-dimensional convergence for the empirical spectral measure  of symmetric and Hermitian matrix processes whose entries are generated from the solution of stochastic differential equation driven by fractional Brownian motion with index $H\in(1/2,1).$

In the present paper, we aim to investigate the existence of multiple spectral collisions for matrix-valued Gaussian processes.  Note that for particle systems, it is natural and interesting to investigate multiple collisions (see, e.g., \cite{Harris65, IK10, BS18}).   Our main objective  is to extend the work \cite{Jaramillo2018}, which dealt with the collision of two eigenvalues ($k=2$ in \eqref{Eq:Question}),   to the collision of multiple eigenvalues  ($k\ge2$) case  and 
also to determine for $\beta \in \{1,2\}$ the Hausdorff dimension of the set ${\cal C}_k^\beta$ of collision times:
\begin{equation}\label{e:h-dim}
{\cal C}_k^\beta  = \{ t \in I: \, \lambda_{i_1}^{\beta}(t) = \cdots = \lambda_{i_k}^{\beta}(t) \mathrm{\ for \ some \ } 
 \, 1 \le i_1 < \cdots < i_k \le d\}.
\end{equation}
The notions of Hausdorff measure, Hausdorff dimension, and capacity will be recalled in the Appendix (Section \ref{sec:Appendix}). 
We refer to Falconer \cite{Falconer2014} or Mattila \cite{mattila1999} for a systematic account on fractal geometry 
and related topics.  

%We assume that there exists a multiparameter index  

We will assume that the associated Gaussian random field $\xi = \{\xi(t): t \in \bR_+^N \}$ satisfies 
the same conditions as in  \cite{Jaramillo2018}.  Namely,  let $(H_1, \ldots, H_N) \in (0,1)^N$  be a constant vector. 
We assume that the following conditions {(A1) } and {(A2) } hold:

\begin{enumerate}
	\item[{(A1) }] There exist  positive and finite constants $c_1$, 
	$c_2$ and $c_3$, such that 
	$\bE[\xi(t)^2] \ge c_1$ for all $t \in I$, and
	\begin{align*}
		c_2 \sum_{j=1}^N |s_j - t_j|^{2H_j} \le \bE \left[ (\xi(s) - \xi(t))^2 \right] \le c_3 \sum_{j=1}^N |s_j - t_j|^{2H_j}
	\end{align*}
	for all $s = (s_1, \ldots, s_N), t = (t_1, \ldots, t_N) \in I$.
	
	\item[{(A2) }] There exists a positive constant $c_4$ such that for all $s , t  \in I$,
	\begin{align*}
		\mathrm{Var} \left[ \xi(t) | \xi(s) \right] \ge c_4 \sum_{j=1}^N |s_j - t_j|^{2H_j},
	\end{align*}
	where $\mathrm{Var} \left[ \xi(t) | \xi(s) \right]$ denotes the conditional variance of $\xi(t)$ given $\xi(s)$.
\end{enumerate}

As an example, we mention that if  $B^H = \{B^H(t): t \in \bR_+^N \}$ is a fractional Brownian motion with index 
$H \in (0, 1)$ which is a centered Gaussian random field with covariance function
\[
C(s,t) = \frac 1 2 \big(\|s\|^{2H} + \|t\|^{2H} - \|s-t \|^{2H}\big), \ \ \ \forall \ s, \, t \in \bR_+^N,
\]
where $\| \cdot \|$ is  the Euclidean norm, then $B^H$ satisfies {(A1) } and {(A2) } with $H_1 = \cdots = H_N = H$.   
Conditions {(A1) } and {(A2) } allow Gaussian random field $\xi = \{\xi(t): t \in \bR_+^N \}$ to be anisotropic in the 
sense that the behavior of $\xi(t)$ may be different in different directions. Examples of such Gaussian random 
fields include the Brownian sheet, fractional Brownian sheets, and the solution to stochastic heat equation.
See \cite{xiao2009sample} for more examples and properties of  Gaussian random fields that satisfy  {(A1) } and {(A2)}.

The following are the main results of this paper.  For the real-valued case, we have
\begin{theorem} \label{Thm-hitting prob-real}
Let $Y^{\beta}$ ($\beta = 1$) be the matrix-valued process defined by  
(\ref{e:Y}) with eigenvalues $\{\lambda_1^\beta(t), \dots, \lambda_d^\beta(t)\}$.  The associated Gaussian random field 
$\xi = \{\xi(t): t \in \bR_+^N \}$ satisfies {(A1) } and {(A2)}. Then for any $k\in\{2,\dots, d\}$ the following statements hold: 
\begin{enumerate}
	\item[(i)] If ~ $\sum_{j=1}^N\frac1{H_j} < (k+2)(k-1)/2$, then
	\begin{align*}
	\bP \left( \lambda_{i_1}^{\beta}(t) = \cdots = \lambda_{i_k}^{\beta}(t) \mathrm{\ for \ some \ } t \in I \ \mathrm{and} \ 
	1 \le i_1 < \cdots < i_k \le d \right) = 0.
	\end{align*}
	\item[(ii)] If~ $\sum_{j=1}^N\frac1{H_j} > (k+2)(k-1)/2$, then
	\begin{align*}
		\bP \left( \lambda_{i_1}^{\beta}(t) = \cdots = \lambda_{i_k}^{\beta}(t) \mathrm{\ for \ some \ } t \in I \ \mathrm{and} \ 
		1 \le i_1 < \cdots < i_k \le d \right) > 0.
	\end{align*}
	\item[(iii)] If~ $\sum_{j=1}^N\frac1{H_j} > (k+2)(k-1)/2$, then with positive probability, the Hausdorff dimension of the 
	set ${\cal C}_k^\beta$ of collision times given in \eqref{e:h-dim} is 
	\begin{equation}\label{Eq:dimC}
	\begin{split}
	\dimh {\cal C}_k^\beta &= \min_{1 \le \ell \le N}\bigg\{\sum_{j=1}^{\ell } \frac{H_\ell}{H_j} + N-\ell - H_\ell \frac{(k+2)(k-1)} 2\bigg\}\\
	& =  \sum_{j=1}^{\ell_0 } \frac{H_{\ell_0}}{H_j} + N-\ell_0 - H_{\ell_0} \frac{(k+2)(k-1)} 2,
	\end{split}
	\end{equation}
	where $\ell_0$ is the smallest $\ell$ such that $\sum_{j=1}^\ell \frac1{H_j}>(k+2)(k-1)/2$, i.e.,  $\sum_{j=1}^{\ell_0 - 1} \frac{1}{H_j} 
	\le \frac{(k+2)(k-1)} 2 < \sum_{j=1}^{\ell_0 } \frac{1}{H_j}$ with the convention $\sum_{j=1}^{0} \frac{1}{H_j}:= 0$.
\end{enumerate}
\end{theorem}

To illustrate this theorem, we consider the special case of a symmetric matrix-valued process $X^{\beta}$ associated with fBm $B^H$.
\begin{corollary}
Let $\beta = 1$ and $Y^{\beta}$ be a  matrix-valued process given in \eqref{e:Y} associated with fBm $B^H$. Then for any 
$k\in\{2,\dots, d\}$ the following hold:

\begin{enumerate}
	\item[(i)] If $N < (k+2)(k-1)H/2$, then
	\begin{align*}
		\bP \left( \lambda_{i_1}^{\beta}(t) = \cdots = \lambda_{i_k}^{\beta}(t) \mathrm{\ for \ some \ } t \in I \ \mathrm{and} \ 
		1 \le i_1 < \cdots < i_k \le d \right) = 0.
	\end{align*}
	\item[(ii)] If $N > (k+2)(k-1)H/2$, then
	\begin{align*}
		\bP \left( \lambda_{i_1}^{\beta}(t) = \cdots = \lambda_{i_k}^{\beta}(t) \mathrm{\ for \ some \ } t \in I \ \mathrm{and} \ 
		1 \le i_1 < \cdots < i_k \le d \right) > 0.
	\end{align*}
	\item[(iii)] If $N > (k+2)(k-1)H/2$, then with positive probability,
	\begin{equation*}
	\dimh {\cal C}_k^\beta =  N - H \frac{(k+2)(k-1)} 2.
	\end{equation*}
\end{enumerate}
\end{corollary}

The following are the results for the complex-valued case.

\begin{theorem} \label{Thm-hitting prob-complex}
Let $Y^{\beta}$ ($\beta = 2$) be the matrix-valued process defined by  
(\ref{e:Y}) with eigenvalues $\{\lambda_1^\beta(t), \dots, \lambda_d^\beta(t)\}$. The associated Gaussian 
random field $\xi = \{\xi(t): t \in \bR_+^N \}$ satisfies {(A1) } and {(A2)}.
Then for any $k\in\{2,\dots, d\}$ the following statements hold:

\begin{enumerate}
	\item[(i)] If~ $\sum_{j=1}^N\frac1{H_j} < k^2 - 1$, then
	\begin{align*}
		\bP \left( \lambda_{i_1}^{\beta}(t) = \cdots = \lambda_{i_k}^{\beta}(t) \mathrm{\ for \ some \ } t \in I \ \mathrm{and} 
		\ 1 \le i_1 < \cdots < i_k \le d \right) = 0.
	\end{align*}
	\item[(ii)] If~ $\sum_{j=1}^N\frac1{H_j} > k^2 - 1$, then
	\begin{align*}
		\bP \left( \lambda_{i_1}^{\beta}(t) = \cdots = \lambda_{i_k}^{\beta}(t) \mathrm{\ for \ some \ } t \in I \ \mathrm{and} 
		\ 1 \le i_1 < \cdots < i_k \le d \right) > 0.
	\end{align*}
	\item[(iii)] If~ $\sum_{j=1}^N\frac1{H_j} > k^2 -1 $, then with positive probability, the Hausdorff dimension of the set 
	${\cal C}_k^\beta$ of collision times given in \eqref{e:h-dim} is 
	\begin{equation*}
	\begin{split}
	\dimh {\cal C}_k^\beta &= \min_{1 \le \ell \le N}\bigg\{\sum_{j=1}^{\ell } \frac{H_\ell}{H_j} + N-\ell - H_\ell \big(k^2 -1\big) \bigg\}\\
	& =  \sum_{j=1}^{\ell_0 } \frac{H_{\ell_0}}{H_j} + N-\ell_0 - H_{\ell_0} \big(k^2 -1\big),
	\end{split}
        \end{equation*}
        where $\ell_0$ is the smallest $\ell$ such that $\sum_{j=1}^{\ell} \frac1{H_j}>k^2-1$. 
	\end{enumerate}
\end{theorem}

\begin{corollary}
Let $\beta = 2$ and $Y^{\beta}$ be a  matrix-valued process given in \eqref{e:Y} associated with fBm $B^H$. Then for any 
$k\in\{2,\dots, d\}$ the following hold:

\begin{enumerate}
	\item[(i)] If $N <(k^2-1)H$, then
	\begin{align*}
		\bP \left( \lambda_{i_1}^{\beta}(t) = \cdots = \lambda_{i_k}^{\beta}(t) \mathrm{\ for \ some \ } t \in I \ \mathrm{and} \ 
		1 \le i_1 < \cdots < i_k \le d \right) = 0.
	\end{align*}
	\item[(ii)] If $N > (k^2-1)H$, then
	\begin{align*}
		\bP \left( \lambda_{i_1}^{\beta}(t) = \cdots = \lambda_{i_k}^{\beta}(t) \mathrm{\ for \ some \ } t \in I \ \mathrm{and} \ 
		1 \le i_1 < \cdots < i_k \le d \right) > 0.
	\end{align*}
	\item[(iii)] If $N > (k^2-1)H$, then with positive probability,
	\begin{equation*}
	\dimh {\cal C}_k^\beta =  N - H (k^2-1).
	\end{equation*}
\end{enumerate}
\end{corollary}

\begin{remark}
In the cases when $\sum_{j=1}^N\frac1{H_j} = (k+2)(k-1)/2$ in Theorem \ref{Thm-hitting prob-real} and when 
$\sum_{j=1}^N\frac1{H_j} = k^2 - 1$ in Theorem \ref{Thm-hitting prob-complex}, it is an open problem 
(except for the matrix-valued processes associated with the Brownian sheet) whether there exist $k$ 
eigenvalues of  $Y^{\beta}$  ($\beta = 1, 2$) that coincide. This is related to the problem on hitting 
probability of Gaussian random fields in critical dimensions, which is still open in general. 
We refer to the seminal paper \cite{khoshnevisan1999brownian} for the resolution of the problem 
for the Brownian sheet and to \cite{dalang2017polarity} for a solution of the problem on the hitting 
probability of a singleton. We believe that 
there is no collision of $k$ eigenvalues of  $Y^{\beta}$ in the critical cases of  $\sum_{j=1}^N\frac1{H_j} 
= (k+2)(k-1)/2$  ($\beta = 1$) and $\sum_{j=1}^N\frac1{H_j} = k^2 - 1$  ($\beta = 2$), respectively. 
However, a rigorous proof would have to rely on new methods that are different from those in the 
present paper. We plan to study this problem in a subsequent project. 
% for the $k$ eigenvalue of $\tilde{x} \in {\mathbf S}(d)$ coincide, existence of 
\end{remark}

\subsection{Notations and preliminaries} \label{sec:notations}

In this subsection, we introduce some notations and preliminaries on matrices that will be used in the proofs. 

For a vector space $\bR^m$ or $\bC^m$, let $\| \cdot \|$ be the Euclidean norm and $\langle \cdot, \cdot \rangle$ be the 
corresponding inner product. For a metric space $X$,  we denote by $\mathfrak B_r(x)$ the open ball  centered at $x\in X$ 
with radius $r$. We also denote by ${\mathbf D}(d)$ the set of diagonal real matrices of dimension $d$. For a matrix $A$, 
denote  by $A^*$ the conjugate of the transpose of $A$. We also denote by $A_{*,j}$  the $j$-th column of $A$. If $A$ is a square 
matrix, then we denote by $\spec(A)$ the spectrum of $A$, i.e. the set of eigenvalues of $A$. We also denote by $\cE^A_{\lambda}$  
the eigenspace associated with $\lambda \in \spec(A)$.

We denote by $\mathbf S(d)$ and $\mathbf H(d)$ the set of real symmetric $d\times d$ matrices and the set of complex 
Hermitian $d\times d$ matrices, respectively. 
By the canonical identification, an element $x \in \bR^{d(d+1)/2}$ is considered the same as  $\tilde{x} = \{\tilde{x}_{i,j}, 1 \le i,j \le d\} 
\in {\mathbf S}(d)$  with $\tilde{x}_{i,j} = x_{i (2d - i + 1)/2 - d + j}, 1 \le i \le j \le d$,  or equivalently, a symmetric matrix
\begin{align*}
	\tilde{x} = \left(
	\begin{matrix}
	\tilde{x}_{11} & \tilde{x}_{12} & \cdots & \tilde{x}_{1d} \\
	\tilde{x}_{12} & \tilde{x}_{22} & \cdots & \tilde{x}_{2d} \\
	\ldots & \ldots &\ldots & \ldots \\
	\tilde{x}_{1d} & \tilde{x}_{2d} & \cdots & \tilde{x}_{dd}
	\end{matrix}
	\right)
\end{align*}
can be viewed as a (unique) row vector $(x_1, x_2, \dots, x_{d(d+1)/2})=(\tilde{x}_{11}, \ldots, \tilde{x}_{1d}, 
\tilde{x}_{22}, \ldots, \tilde{x}_{2d}, \ldots, \tilde{x}_{dd})$.
In a similar way, one can identify $x \in \bR^{d^2}$ with $\tilde{x} 
\in {\mathbf H}(d)$ whose entries are 

\begin{align*}
	\tilde{x}_{i,j} =
	\begin{cases}
	x_{(i-1)(2d-i+2)/2 +1 }, & i=j, \\
	x_{(i-1)(2d-i+2)/2 +1 + j-i} + \iota x_{d(d+1)/2 + (i-1)(2d-i)/2 +j-i}, & i<j,
	\end{cases}
\end{align*}
or equivalently, a Hermitian matrix
\begin{align*}
	\tilde{x} = \left(
	\begin{matrix}
	\tilde{x}_{11} & \tilde{x}_{12} & \cdots & \tilde{x}_{1d} \\
	\overline{\tilde{x}_{12}} & \tilde{x}_{22} & \cdots & \tilde{x}_{2d} \\
	\ldots & \ldots &\ldots & \ldots \\
	\overline{\tilde{x}_{1d}} & \overline{\tilde{x}_{2d}} & \cdots & \tilde{x}_{dd}
	\end{matrix}
	\right) 
\end{align*}
can be understood  as  a (unique) row vector 
\begin{align*}
&(x_1,x_2,\dots, x_{d^2})\\
&= \Big(\tilde{x}_{11}, \mathrm{Re}(\tilde{x}_{12}), \ldots, \mathrm{Re}(\tilde{x}_{1d}), \tilde{x}_{22}, \mathrm{Re}(\tilde{x}_{23}), 
\ldots, \mathrm{Re}(\tilde{x}_{2d}), \ldots, \tilde{x}_{dd}, \\
&~~~~~ \qquad \qquad \qquad \quad \mathrm{Im}(\tilde{x}_{12}), \ldots, \mathrm{Im}(\tilde{x}_{1d}), \mathrm{Im}(\tilde{x}_{23}), 
\ldots, \mathrm{Im}(\tilde{x}_{(d-1)d})\Big).
\end{align*}
   
Throughout the paper, for a vector $x$ in  $\bR^{d(d+1)/2}$ ($\bR^{d^2}$, resp.), the symbol $\tilde x$ 
means the corresponding matrix in ${\mathbf S}(d)$ (${\mathbf H}(d)$, resp.).  We now introduce some other notations. 

For a vector $x$ in $\bR^{d(d+1)/2}$ or $\bR^{d^2}$,  let $E_i(x)$ be the $i$-th smallest eigenvalue of $\tilde{x}$. 
Then $E_i(x)$ is a continuous function of $x$ for each $i\in\{1,\dots, d\}$, noting that $(E_1(x), \dots, E_d(x))$ 
are ordered roots of the characteristic polynomial of $\tilde x$. 

For $k\in\{1,\dots, d\}$, define
\begin{align} \label{def-degenerate-real}
	{\mathbf S}(d;k) = \{x \in\bR^{d(d+1)/2}  | \ E_{i+1} (x) = E_{i+2}(x)=\cdots = E_{i+k}(x), \mathrm{\ for \ some \ } 0 \le i\le d-k \}
\end{align}
%to be the set of $x \in\bR^{d(d+1)/2}$ such that at least $k$ eigenvalue of $\tilde{x} \in {\mathbf S}(d)$ coincide, 
and %its Hermitian analogue
\begin{align} \label{def-degenerate-complex}
	{\mathbf H}(d;k) = \{x \in \bR^{d^2} \ | \ E_{i+1} (x) =E_{i+2}(x)= \cdots = E_{i+k}(x), \mathrm{\ for \ some \ } 0 \le i  \le d-k \}.
\end{align}
Due to the canonical identification between vectors and matrices mentioned above, we also regard the set  ${\mathbf S}(d;k)$ 
(${\mathbf H}(d;k)$, reps.) of  vectors  as the set of $d\times d$ symmetric (Hermitian, resp.) matrices, each element of which 
has at least $k$ identical eigenvalues.

For  $m,\, n\in\mathbb N$, let $\bR^{m\times n}$ ($\bC^{m\times n}$, resp.) be the space of $m\times n$ real (complex, resp.)
matrices, and  we take the Frobenius norm, i.e., for $A\in \bC^{m\times n}$,
\begin{equation}\label{eq:F-norm}
\|A\|=\bigg(\sum_{i=1}^m\sum_{j=1}^n |A_{ij}|^2 \bigg)^{1/2}.
\end{equation}
Thus $\|A\|$ is just the Euclidean norm of $A$,  if we consider $A$ as a vector of size $m\cdot n$.

For $l\in\{0,1, \dots, d-1\}$, define  
\begin{align} \label{def-O(d;k)}
	{\mathbf O}(d;l) = \{A \in \bR^{d \times (d-l)}: A^*A = I_{d-l}\},
\end{align}
and
\begin{align} \label{def-U(d;k)}
	{\mathbf U}(d;l) = \{A \in \bC^{d \times (d-l)}: A^*A = I_{d-l}\},
\end{align}
recalling that $A^*$ is the conjugate of the transpose of $A$.  In particular, for the case $l=0$,  we denote 
${\mathbf O}(d) = {\mathbf O}(d;0)$ and ${\mathbf U}(d) = {\mathbf U}(d;0)$, which are the set of $d\times d$ 
orthogonal matrices and the set of $d\times d$ unitary matrices, respectively.

By the regular level set theorem (\cite[Theorem 9.9]{Tu2011}, \cite[Theorem 4.2]{Jaramillo2018}), one can show that 
(see, e.g., \cite[page 7]{Jaramillo2018}) ${\mathbf O}(d;l)$  is a smooth submanifold of $\bR^{d(d-l)}$ 
of dimension $d(d-l)-\frac1{2}(d-l)(d-l+1)= \frac12[d(d-1) -l(l-1)]$ and ${\mathbf U}(d;l)$ is a smooth submanifold of 
$\bC^{d (d-l)} \cong \bR^{2d(d-l)}$ of dimension $2d(d-l)-(d-l)^2=d^2 - l^2$,  and therefore,
\[{\mathbf O}(d;l)\cong \bR^{\frac12[d(d-1)-l(l-1)]}; ~~{\mathbf U}(d;l)\cong \bR^{d^2-l^2}, ~\text{ for } l\in\{0, 1, \dots, d-1\}.\]

%\medskip
%
%The rest of the paper is organized as follows. In Section \ref{sec:real-case}, we prove Theorem \ref{Thm-hitting prob-real} for the real case, 

\section{The real case: proof of Theorem \ref{Thm-hitting prob-real}}\label{sec:real-case}

In this section, we deal with the matrix \eqref{def-entries} for $\beta = 1$ and prove Theorem \ref{Thm-hitting prob-real}.

First, we provide an upper bound for the dimension of $\mathbf S(d;k)$ given in \eqref{def-degenerate-real}. It is clear 
that the set ${\mathbf S}(d)=\mathbf S(d;1)$ of $d\times d$ real symmetric matrices  has dimension $d(d+1)/2$. For 
the dimension of ${\mathbf S}(d;k), k=1,\dots, d$, the upper bound $\frac12[d(d+1) - k(k+1)]+1$  is a direct consequence of  
Lemma \ref{Lem:dim-S} below. Note that this result was obtained in \cite[Proposition 4.5]{Jaramillo2018} for the case $k=2$. 

\begin{lemma} \label{Lem:dim-S}
(i) Fix $k\in\{1, \dots, d-1\}$. Let $\Delta: \bR^{d-k+1} \rightarrow {\mathbf D}(d)$ be a function that maps each vector 
$u= (u_1, \ldots, u_{d-k+1}) \in \bR^{d-k+1}$
 to a $d\times d$ diagonal matrix $\Delta(u)$ with entries given by
\begin{align} \label{def-Lamda mapping}
	\Delta_{i,i}(u) =
	\begin{cases}
	u_i, & 1 \le i \le d-k; \\
	u_{d-k+1}, & d-k+1 \le i \le d.
	\end{cases}
\end{align}
Then there exists a compactly supported smooth function $\Gamma: \bR^{\frac12[d(d-1) - k(k-1)]} \rightarrow \bR^{d \times d}$, 
such that the mapping \[G: \bR^{d-k+1}\times   \bR^{\frac12[d(d-1) - k(k-1)]} \rightarrow {\mathbf S}(d)\] given by
\begin{align} \label{def-F function}
	G(u, v) = \Gamma(v) \Delta(u) \Gamma(v)^*,~~ (u,v) \in \bR^{d-k+1}\times  \bR^{\frac12[d(d-1) - k(k-1)]},
\end{align}
satisfies
\begin{align} \label{eq-Lem:dim-S}
	{\mathbf S}(d;k) \subseteq \left\{ x \in \bR^{d(d+1)/2}: \ \tilde{x} \in \mathrm{Im}(G) \right\},
\end{align}
where $\im(G)$ is the image of the mapping $G$.

(ii) If $k=d$, the dimension of ${\mathbf S}(d;d)$ is 1.
\end{lemma}

\begin{proof}
It is easier to prove (ii). If $k=d$, then ${\mathbf S}(d;d)$ consists of all symmetric matrices whose eigenvalues are 
all the same, and ${\mathbf S}(d;d)$ has dimension 1. 

Next we prove (i). For $k\in\{1,\dots, d-1\}$, we will follow the approach used in the proof of  \cite[Proposition 4.5]{Jaramillo2018}, 
which is based on the Gram-Schmidt orthonormalisation applied to manifolds.

For $\epsilon > 0$, let $I_{\epsilon}^{(k)}$ denote the following open interval
\begin{equation}\label{e:Je}
I_{\epsilon}^{(k)} = (-\epsilon, \epsilon)^{\frac12[d(d-1) - k(k-1)]}\,.
\end{equation}

Recall that $\mathbf O(d;k)$ is defined in \eqref{def-O(d;k)}.  For an arbitrary fixed matrix $A \in {\mathbf O}(d;k)$, 
the columns $\{A_{*,1}, \ldots, A_{*,d-k}\}$ of $A$ are orthonormal,  and hence we can extend them to an 
orthonormal basis of $\bR^d$. Thus, there exists an orthogonal matrix $\hat A \in {\mathbf O}(d)$, such that 
$\hat A_{*,j} = A_{*,j}$ for all $1 \le j 
\le d-k$.  

Also recalling (Section \ref{sec:notations})   that ${\mathbf O}(d;k)$ is a smooth manifold of dimension $\frac12[d(d-1) - k(k-1)]$, 
by the definition of chart (see e.g. \cite[Definition 5.1]{Tu2011}), there exists a neighbourhood $U(A)$ of $A$ in ${\mathbf O}(d;k)$ 
such that $U(A)$  is smoothly diffeomorphic to $I_{\epsilon}^{(k)}$. Namely,  we can choose $0<\delta<\frac1{4d}$ 
such that there exists a positive number $\epsilon$ (which may depend on $A$) and a smooth diffeomorphism
\begin{align} \label{def-family of varphi}
	\phi: I_{\epsilon}^{(k)} \rightarrow U(A)\subset {\mathbf O}(d;k) \cap \mathfrak B_{\delta}(A)
\end{align}
satisfying $\phi(0) = A$. Here $\mathfrak B_\delta(A)$ is the the open ball with radius $\delta$ centered at $A$ in the space 
$\bR^{d\times (d-k)}$ of $d\times (d-k)$ matrices 
under the Frobenius norm $\|\cdot\|$ given by \eqref{eq:F-norm}.   Thus, for any matrix $B \in U(A)$, we have $\|A-B\|<\delta$,  
and hence $\| \hat A_{*,j} -B_{*,j} \|=\|  A_{*,j} -B_{*,j} \| < \delta$ for  $1 \le j \le d-k$. Furthermore,
\begin{align*}
	&\quad \left| \Bigg\| \hat A_{*,d-k+1} - \sum_{j=1}^{d-k} \langle B_{*,j}, \hat A_{*,d-k+1} \rangle B_{*,j} \Bigg\| - 1 \right| \\
	&= \left| \Bigg\| \hat A_{*,d-k+1} - \sum_{j=1}^{d-k} \langle B_{*,j}, \hat A_{*,d-k+1} \rangle B_{*,j} \Bigg\| - \Bigg\| \hat A_{*,d-k+1} 
	- \sum_{j=1}^{d-k} \langle \hat A_{*,j}, \hat A_{*,d-k+1} \rangle \hat A_{*,j} \Bigg\| \right| \\
	&\le \left\| \sum_{j=1}^{d-k} \langle B_{*,j}, \hat A_{*,d-k+1} \rangle B_{*,j} - \sum_{j=1}^{d-k} \langle \hat A_{*,j}, \hat A_{*,d-k+1} \rangle \hat A_{*,j} \right\| \\
	&\le \left\| \sum_{j=1}^{d-k} \langle B_{*,j}, \hat A_{*,d-k+1} \rangle (B_{*,j} - \hat A_{*,j}) \right\| + \left\| \sum_{j=1}^{d-k} \langle B_{*,j} 
	- \hat A_{*,j}, \hat A_{*,d-k+1} \rangle \hat A_{*,j} \right\| \\
	&\le 2(d-k)\delta < \dfrac{1}{2}.
\end{align*}
In the above we have used the orthonormality of the family $\{\hat A_{*,i}\}_{1 \le i \le d}$ as well as the triangle inequality. 
Hence, we can see that $\| \hat A_{*,d-k+1} - \sum_{j=1}^{d-k} \langle \phi_{*,j}(v), \hat A_{*,d-k+1} \rangle \phi_{*,j}(v) \|$ 
is bounded away from zero for all $v \in I_{\epsilon}^{(k)}$, where $\phi_{*,j}(v)$ is the $j$-th column vector of $\phi(v)$ for 
$j=1, \dots, d-k$. Thus,  the following mapping, for $v\in I_\epsilon^{(k)}$,
\begin{align} \label{eq-origin-3.3}
	\varphi_{k-1}(v) = \dfrac{\hat A_{*,d-k+1} - \sum_{j=1}^{d-k} \langle \phi_{*,j}(v), \hat A_{*,d-k+1} \rangle \phi_{*,j}(v)}
	{\| \hat A_{*,d-k+1} - \sum_{j=1}^{d-k} \langle \phi_{*,j}(v), \hat A_{*,d-k+1} \rangle \phi_{*,j}(v) \|}
\end{align}
is well-defined and smooth. Note that $\| \varphi_{k-1}-\hat A_{*,d-k+1}\|$ could be arbitrarily small as $\delta$ goes to 0. 
Hence, similarly,   the mappings,  for $0 \le l \le k-2$ and 
$v\in I_\epsilon^{(k)}$,
\begin{align} \label{def-family of psi}
	\varphi_l(v) = \dfrac{\hat A_{*,d-l} - \sum_{i=l+1}^{k-1} \langle \varphi_i(v), \hat A_{*,d-l} \rangle \varphi_i(v) - 
	\sum_{j=1}^{d-k} \langle \phi_{*,j}(v), \hat A_{*,d-l} \rangle \phi_{*,j}(v)}{\| \hat A_{*,d-l} - 
	\sum_{i=l+1}^{k-1} \langle \varphi_i(v), \hat A_{*,d-l} \rangle \varphi_i(v) - \sum_{j=1}^{d-k} \langle \phi_{*,j}(v), 
	\hat A_{*,d-l} \rangle \phi_{*,j}(v) \|}
\end{align}
are well-defined and smooth, if $\delta$ is taken sufficiently small. Note that \eqref{eq-origin-3.3} is included in
\eqref{def-family of psi} as the case $l = k-1$.  By the construction (the Gram-Schmidt orthonomalization), one can 
verify that the family of vectors $\{\phi_{*,1}(v), \ldots, \phi_{*,d-k}(v), \varphi_{k-1}(v), \ldots, \varphi_0(v)\}$ are 
orthonormal for all $v \in I_{\epsilon}^{(k)}$.

Therefore, we may construct a smooth function \[\Gamma: \bR^{\frac12[d(d-1) - k(k-1)]} \rightarrow \bR^{d \times d}\] with compact  
support such that for $v\in I_{\epsilon}^{(k)}$
\begin{align} \label{def-mapping Pi}
	\Gamma_{*,j}(v) =
	\begin{cases}
	\phi_{*,j}(v), & 1 \le j \le d-k, \\
	\varphi_{d-j}(v), & d-k+1 \le j \le d.
	\end{cases}
\end{align}
 Recall that $\phi$ in (\ref{def-family of varphi}) is a diffeomorphism,  and the set
\begin{align}\label{eq:open-cover}
	U_{\Gamma}(A) = \left\{ \left( \Gamma_{*,1}(v), \ldots, \Gamma_{*,d-k}(v) \right) : v \in I_{\epsilon}^{(k)} \right\} 
	= \phi(I_{\epsilon}^{(k)})=U(A)
\end{align}
is an open subset of ${\mathbf O}(d;k)$  containing $A$. Hence, the collection of the sets $\{U_{\Gamma}(A): A \in {\mathbf O}(d;k)\}$ 
forms an open cover for ${\mathbf O}(d;k)$.

Due to the compactness of ${\mathbf O}(d;k)$,  one can find a finite number of open covers $\{U_{\Gamma^{(i)}}(A_i), i=1, \dots, M\}$ 
of the form of \eqref{eq:open-cover} for some $M\in \bN$,   such that
\begin{align} \label{eq-finite open cover}
	{\mathbf O}(d;k) = \bigcup_{i=1}^M U_{\Gamma^{(i)}}(A_i),
\end{align}
where $A_1, \dots, A_M$ are distinct matrices in $\mathbf O(d;k)$, $\Gamma^{(1)}, \ldots, \Gamma^{(M)}$ are smooth mappings 
of the form \eqref{def-mapping Pi} supported in the intervals $I_{\epsilon_1}^{(k)}, \ldots, I_{\epsilon_M}^{(k)}$ respectively, and 
\begin{align*}
	U_{\Gamma^{(i)}}(A_i) = \left\{ \left( \Gamma_{*,1}^{(i)}(v), \ldots, \Gamma_{*,d-k}^{(i)}(v) \right) : v \in I_{\epsilon_i}^{(k)} \right\}.
\end{align*}
For $1 \le i\le M$, we define mappings $G^{(i)}: \bR^{d-k+1}\times   I_{\epsilon_i}^{(k)} \rightarrow {\mathbf S}(d)$ by
\begin{align*}
	G^{(i)}(u,v) = \Gamma^{(i)}(v) \Delta(u) \Gamma^{(i)}(v)^*,
\end{align*}
for $v \in I_{\epsilon_i}^{(k)}$ and $u \in \bR^{d-k+1}$. 

For an arbitrary fixed $x \in {\mathbf S}(d;k)$, we have the decomposition $\tilde{x} = Q D Q^*$ for some $Q \in {\mathbf O}(d)$ and 
$D \in {\mathbf D}(d)$.  We assume that the last $k$ eigenvalues are the same, i.e., 
$D_{d-k+1, d-k+1} = \cdots = D_{d,d}$, by rearranging the diagonal of $D$ and the columns of $Q$ if necessary. Note that the matrix 
$(Q_{*,1}, \ldots, Q_{*,d-k}) \in {\mathbf O}(d;k)$. Thus by \eqref{eq-finite open cover},  there exists $ i_0 \in\{1, \dots,  M\}$ such that 
$(Q_{*,1}, \ldots, Q_{*,d-k}) \in U_{\Gamma^{(i_0)}}(A_{i_0})$, and  hence one can find $v \in I_{\epsilon_{i_0}}^{(k)}$ such that 
$(Q_{*,1}, \ldots, Q_{*,d-k}) = (\Gamma_{*,1}^{(i_0)}(v),  \ldots, \Gamma_{*,d-k}^{(i_0)}(v))$. By the 
construction \eqref{def-mapping Pi}, both $\{ \Gamma_{*,1}^{(i_0)}(v), \ldots, \Gamma_{*,d}^{(i_0)}(v) \}$ and $\{Q_{*,1},
\ldots, Q_{*,d}\}$ are orthonormal bases of $\bR^d$ with the first $d-k$ vectors coinciding, and hence $\{ \Gamma_{*,1}^{(i_0)}(v), 
\ldots, \Gamma_{*,d}^{(i_0)}(v) \}$ also forms  a basis of eigenvectors for $\tilde x$. 
 
% Thus, as the orthogonal 
% complement of the space $\mathrm{span} \{Q_{*,1}, \ldots, Q_{*,d-k}\}$, we have
%\begin{align*}
%	\mathrm{span} \left\{ \Gamma_{*,d-k+1}^{(i_0)}(v), \ldots, \Gamma_{*,d}^{(i_0)}(v) \right\}
%	= \mathrm{span} \left\{ Q_{*,d-k+1}, \ldots, Q_{*,d} \right\}.
%\end{align*}
%Hence, $\{ \Gamma_{*,d-k+1}^{(i_0)}(v), \ldots, \Gamma_{*,d}^{(i_0)}(v) \}$ forms an orthonormal basis of the 
%eigenspace $\cE_{D_{d-k+1,d-k+1}}^{\tilde{x}}$ associated with $D_{d-k+1,d-k+1}$. 

Let  $u = (D_{1,1}, \ldots, D_{d-k+1,d-k+1}) \in \bR^{d-k+1}$ then $D = \Delta(u)$.
 Thus, $\tilde{x}$ 
has the decomposition
\begin{align*}
	\tilde{x}
	= \Gamma^{(i_0)}(v) \Delta(u) \Gamma^{(i_0)}(v)^*
	= G^{(i_0)} (u,v).
\end{align*}
Since $x\in \mathbf S(d;k)$ is arbitrarily chosen, we conclude that
\begin{align*}
	{\mathbf S}(d;k) \subseteq \left\{ x \in \bR^{d(d+1)/2}: \tilde{x} \in \bigcup_{l=1}^M \mathrm{Im}(G^{(l)}) \right\}.
\end{align*}

Finally, let $\epsilon=\max\{\epsilon_1, \dots, \epsilon_M\}$, then for any smooth function $\Gamma$ supported on 
$I_{3M\epsilon}^{(k)}$ satisfying
\begin{align*}
	\Gamma(y) = \Gamma^{(j+1)} \left( y - (3j\epsilon, 0, \dots, 0)\right), 
\end{align*}
for $y \in \mathfrak B_{\epsilon} \left( 3 j \epsilon, 0, \ldots, 0 \right) 
	\subseteq \bR^{\frac12[d(d-1) - k(k-1)]}$  for some  $j\in\{0,1,\dots, M-1\}$,
the mapping $G$ defined by \eqref{def-F function} satisfies \eqref{eq-Lem:dim-S}.
\end{proof}

Consider $\mathbf S(d;k)$ defined in \eqref{def-degenerate-real} as the set of $d\times d$ symmetric matrices which 
have at least $k$ identical eigenvalues.  The following lemma, which is an extension of \cite[Lemma 4.3]{Jaramillo2018}, 
claims that the eigenvectors of matrices in $\mathbf S(d;k)$ are continuous at the matrices with  $d-k+1$ distinct 
eigenvalues. 

\begin{lemma} \label{Lem:approx-S}
Fix $k\in\{1,\dots, d\}$. Let $A \in {\mathbf S}(d;k)$ be a symmetric matrix with decomposition $A = P D P^*$ for 
some $P \in {\mathbf O}(d)$ and 
$D \in {\mathbf D}(d)$, such that $|\spec(A)| = d-k+1$. Then for any $\epsilon > 0$, there exists $\delta > 0$, such that 
for all $B \in {\mathbf S}(d;k)$ satisfying
\begin{align*}
	\max_{1 \le i, j \le d} |A_{i,j} - B_{i,j}| < \delta,
\end{align*}
we have  $|\spec(B)|=d-k+1$, and there exists a spectral decomposition  $B = Q FQ^*$, where $Q \in {\mathbf O}(d)$ 
and $F \in {\mathbf D}(d)$ satisfy
\begin{align} \label{eq:eigen-continuity}
\max_{1 \le i \le d} |D_{i,i} - F_{i,i}| < \epsilon, \quad 	\max_{1 \le i, j \le d} |Q_{i,j} - P_{i,j}| < \epsilon.
\end{align}
%{\blue Furthermore, for $k\in\{1, \dots, d-1\}$,  if $\epsilon$ is sufficiently small, then for any two  decompositions 
%$B=QEQ^*$ and $B=Q'E'Q'^*$ satisfying \eqref{eq:eigen-continuity}, we must have $E=E'$ and the eigenvectors in  
%$Q$ and $Q'$ associated with the $d-k$ distinct eigenvalues coincide. }
\end{lemma}

\begin{proof} The proof of this lemma is similar to that of \cite[Lemma 4.3]{Jaramillo2018}. We include it for the reader's convenience.

The first inequality of \eqref{eq:eigen-continuity}, which describes the continuity of the eigenvalues, follows directly from the 
continuity of the functions $E_1, \ldots, E_d$ which are introduced in Section \ref{sec:notations}. The second inequality of \eqref{eq:eigen-continuity} claims that  
eigenvectors, considered as  functions of matrices in $\mathbf S(d;k)$, are continuous at $A\in \mathbf S(d;k)$ with 
$\spec(A)=d-k+1$. The key idea to prove this is to represent  eigenprojections as  matrix-valued Cauchy integrals.

 Noting that $A\in \mathbf S(d;k)$ and 
$|\spec(A)| = d-k+1$, without loss of generality we  assume 
\begin{align} \label{eq-order of eigenvalue}
D_{1,1} < \cdots < D_{d-k+1,d-k+1} = \cdots = D_{d,d}.
\end{align}
For $i = 1, \ldots, d-k+1$, let $\cC_i^A \subseteq \bC \setminus \spec(A)$ be any smooth closed curve around $D_{i,i}$ 
and denote by $\cI_i^A$ the closure of the interior of $\cC_i^A$. We can choose the curves $\{\cC_i^A: 1 \le i \le d-k+1\}$ 
with sufficiently small diameters so that $\cI_1^A, \ldots, \cI_{d-k+1}^A$ are disjoint. For simplicity, let $\cC_i^A = 
\cC_{d-k+1}^A$ and $\cI_i^A = \cI_{d-k+1}^A$ for $d-k+1 < i \le d$.

For $\delta > 0$, we define
\begin{align*}
	U_{\delta} = \left\{ B \in {\mathbf S}(d;k): \max_{1 \le i, j \le d} |A_{i,j} - B_{i,j}| < \delta  \right\}.
\end{align*}
By the continuity of the functions $E_1, \ldots, E_{d-k+1}$ and \eqref{eq-order of eigenvalue}, there exists $\delta > 0$, 
such that for all $B \in U_{\delta}$, we have $E_1(B) < \cdots < E_{d-k+1}(B)$ 
and $E_i(B) \in \cI_i^A$ for $1 \le i \le d-k+1$. 
Noting that  $U_{\delta} \subseteq {\mathbf S}(d;k)$, we have
\begin{align} \label{eq-strictly order of eigenvalue}
E_1(B) < \cdots < E_{d-k+1}(B) = \cdots = E_d(B), \quad \forall B \in U_{\delta},
\end{align}
and
\begin{align} \label{eq-region of eigenvalue}
	E_i(B) \in \cI_i^A, \quad \forall B \in U_{\delta},~ 1 \le i \le d.
\end{align}
For $i=1, \dots, d,$ define the mappings $\
\Theta_i^A: U_{\delta} \rightarrow {\mathbf S}(d)$ by the following matrix-valued Cauchy integrals:
\begin{align*}
	\Theta_i^A(B) = \dfrac{1}{2\pi\iota} \oint_{\cC_i^A} (\zeta I_d - B)^{-1} d\zeta\,.
\end{align*}
Then they are continuous with respect to  $B$ for $B\in U_\delta$. By \cite[page 200, Theorem 6]{Lax2002}, the matrix 
$\Theta_i^A(B)$ is a projection over the sum of the eigenspaces associated with eigenvalues of $B$ that are inside 
$\cI_i^A$. Hence, by \eqref{eq-strictly order of eigenvalue} and \eqref{eq-region of eigenvalue}, 
$\Theta_i^A(B)$ is a projection over the eigenspace $\cE_{E_i(B)}^B$ for $1 \le i \le d$, noting that $\cI_1^A, \dots, \cI_{d-k+1}^A$ are disjoint.

For $1 \le j \le d-k$, we define
\begin{align} \label{eq-def of w1}
	w^j = \dfrac{\Theta_j^A(B) P_{*,j}}{\left\| \Theta_j^A(B) P_{*,j} \right\|}\,,
\end{align}
which clearly are unit eigenvectors of $E_j(B)=F_{j,j}$ for $j=1,\dots, d-k$, noting that the matrix $\Theta_j^A(B)$ is 
a projection over $\mathcal E_{E_j(B)}^B$.  For $d-k+1 \le j \le d$, we define iteratively for $B\in U_\delta$ with $\delta$
 being sufficiently small, by applying the Gram-Schmidt orthonormalizing process to the linearly independent set  
 $\Big\{\Theta_{d-k+1}^A(B) P_{*,d-k+1}, \cdots, \Theta_{d}^A(B) P_{*,d}\Big\}$: 
\begin{equation} \label{eq-def of w2}
\begin{cases}
w^{d-k+1} = \dfrac{\Theta_{d-k+1}^A(B) P_{*,d-k+1}}{\left\| \Theta_{d-k+1}^A(B) P_{*,d-k+1} \right\|}~,\\
\\
	w^j = \dfrac{\dfrac{\Theta_j^A(B) P_{*,j}} {\big\| \Theta_j^A(B) P_{*,j} \big\|} - \sum_{i=d-k+1}^{j-1}
	 \Big\langle \dfrac{\Theta_j^A(B) P_{*,j}}  {\| \Theta_j^A(B) P_{*,j} \|}, w^i \Big\rangle w^i}
	 {\bigg\| \dfrac{\Theta_j^A(B) P_{*,j}} {\| \Theta_j^A(B) P_{*,j} \|} - \sum_{i=d-k+1}^{j-1} 
	 \Big\langle \dfrac{\Theta_j^A(B) P_{*,j}} {\| \Theta_j^A(B) P_{*,j} \|}, w^i \Big\rangle w^i \bigg\|}~, \quad  j=d-k+2, \dots, d,
	 \end{cases}
\end{equation}
which are unit eigenvectors of $E_j(B)=F_{j,j}=F_{d-k+1,d-k+1}$ for $j=d-k+1, \dots, d$, noting that 
$\mathcal E_{E_{d-k+1}(B)}^B=\dots=\mathcal E_{E_{d}(B)}^B$ is a $k$-dimensional vector space.

Recall that $\Theta_j^A(B)$ is a continuous function of $B$ for $B\in U_\delta$ with $\delta$ sufficiently small and  
that $\Theta_j^A(A) P_{*,j} = P_{*,j}$,  for all $1 \le j \le d$.  Also note that the inner product $\langle u, v\rangle$ is 
a continuous function of $(u, v)$, and hence $\omega^j, j=1,\dots, d$, defined by \eqref{eq-def of w1} and \eqref{eq-def of w2}  
are continuous functions of $B$ in a sufficiently small neighborhood $U_\delta$ of $A$.  Thus,  for any $\epsilon>0$, 
one can find a sufficiently small positive constant $\delta$, such that for all $B\in U_\delta$, 
\begin{align*}
	\max_{1 \le j \le d} \| P_{*,j} - w^i\| < \epsilon.
\end{align*}
Thus if we denote the matrix $Q=[\omega^1,\dots, \omega^d]$, then $B=QFQ^*$ and the the second inequality 
of \eqref{eq:eigen-continuity} is satisfied.  The proof is concluded.  
\end{proof}

The following result is an extension of \cite[Proposition 4.7]{Jaramillo2018}, and it shows that  $(\frac12[d(d+1) - k(k+1)]+1)$ 
is an optimal upper bound for the dimension of $\mathbf S(d;k)$.

\begin{lemma} \label{Lem:dim-bound}
Let $k\in\{1,\dots, d\}$. For  $x_0 \in {\mathbf S}(d;k)$ with $|\spec(\tilde{x}_0)| = d-k+1$, there exists $\delta_0 > 0$
such that ${\mathbf S}(d;k) \cap\mathfrak B_{\delta_0}(x_0)$ is a $(\frac12[d(d+1) - k(k+1)]+1)$-dimensional manifold. 
In particular, ${\mathbf S}(d;k) \cap \mathfrak B_{\delta_0}(x_0)$ has positive $(\frac12[d(d+1) - k(k+1)]+1)$-dimensional 
Lebesgue measure.
\end{lemma}

\begin{proof}
The result is obvious if $k=d$. Now we prove for the case $k\in\{1,\dots, d-1\}$.

Let $P \in {\mathbf O}(d)$ and $D \in {\mathbf D}(d)$ such that $\tilde{x}_0$ has the decomposition $\tilde{x}_0 = PDP^*$. 
Since $x_0 \in {\mathbf S}(d;k)$ and $|\spec(\tilde{x}_0)| = d-k+1$,  we assume without loss of generality that 
$D_{1,1} < \cdots < D_{d-k+1, d-k+1} = \cdots =D_{d,d}$. 

Denote by $A \in {\mathbf O}(d;k)$ the 
matrix obtained from $P$ by deleting the last $k$ columns.  For $\epsilon> 0$, recall that $I_{\epsilon}^{(k)}$ is defined 
in \eqref{e:Je}. Then similar to the proof of Lemma \ref{Lem:dim-S}, 
one can show the orthogonality of  the vectors $\phi_{*,1}(v), \ldots, \phi_{*,d-k}(v), 
\varphi_{k-1}(v), \ldots, \varphi_0(v)$,  where $\phi: I_\epsilon^{(k)} \to {\mathbf O}(d;k)\cap \mathfrak B_{\gamma}(A)$ 
with $\gamma\in (0, \sqrt 2/2)$ being sufficiently small  is a diffeomorphism   and  $\{\varphi_j, j=0, \dots, k-1\}$ are smooth functions given by \eqref{def-family of psi}.
%Next, we define $\Gamma': I_{\epsilon}^{(k)} \rightarrow {\mathbf O}(d)$ by
%\begin{align} \label{eq-def Pi'}
%	\Gamma_{*,j}'(\alpha) =
%	\begin{cases}
%		\phi_{*,j}(\alpha), & 1 \le j \le d-k; \\
%		\psi_{d-j}(\alpha), & d-k+1 \le j \le d,
%	\end{cases}
%\end{align}
Define $\bar G: \bR^{d-k+1}\times   I_{\epsilon}^{(k)} \rightarrow {\mathbf S}(d;k)$ by
 \[\bar G(u,v) = \Gamma(v) \Delta(u) \Gamma(v)^*,\]
where $\Gamma: I_\epsilon^{(k)}\to \mathbf O(d)$ is given in \eqref{def-mapping Pi}. In particular, recall that $\phi(0)=A$ 
and hence $\tilde x_0=G(u_0,0)$ where $u_0\in\bR^{d-k+1}$ consists of the $d-k+1$ distinct eigenvalues of $\tilde x_0$.

To show that the manifold ${\mathbf S}(d;k) \cap \mathfrak B_{\delta_0}(x_0)$ has dimension $\frac12[d(d+1)- k(k+1)]+1$ 
and  has positive  $(\frac12[d(d+1)- k(k+1)]+1)$-dimensional Lebesgue measure, it is sufficient 
to show that there exist open sets $U \subseteq  \bR^{d-k+1}\times   I_{\epsilon}^{(k)}$ and $V \subseteq {\mathbf S}(d;k)\cap 
\mathfrak B_{\delta_0}(x_0)$, such that the map $\bar G|_U: U \rightarrow V$ is a homeomorphism.

Let 
\begin{align}\label{e:r0}
	r_0 = \dfrac{1}{2} \min_{\substack{\mu , \lambda\in \spec(\tilde{x}_0)\\ \mu\not= \lambda}} |\mu - \lambda|.
\end{align}
By Lemma \ref{Lem:approx-S}, for some fixed $\gamma_0\in(0, \gamma)$, there exists $\delta_0$ such that for each 
$x \in {\mathbf S}(d;k) \cap \mathfrak B_{\delta_0}(x_0)$,  it has the decomposition $\tilde{x} = Q E Q^*$ with $Q \in 
{\mathbf O}(d) \cap \mathfrak B_{\gamma_0}(P)$ and $E \in {\mathbf D}(d) \cap \mathfrak B_{r_0}(D)$.

Let $u_j=E_{jj}$ for $j=1, \dots, d$ and denote $u = (u_1, \dots, u_{d-k+1}) \in \bR^{d-k+1}$.  Then by the definition 
\eqref{e:r0} of $r_0$ and the fact $E\in \mathbf D(d)\cap \mathfrak B_{r_0}(D)$,  we have $u_1<u_2<\dots<u_{d-k+1}
=\dots=u_{d}$. Furthermore,  for the eigenvectors associated with the $d-k$ non-repeated eigenvalues,  $(Q_{*,1}, \ldots, 
Q_{*,d-k})$ belongs to $\mathbf O(d ;k)\cap \mathfrak B_{\gamma_0}(A)$ noting that $Q \in {\mathbf O}(d) \cap 
\mathfrak B_{\gamma_0}(P)$. Thus,  as in the proof of Lemma \ref{Lem:dim-S}, there exists $v \in I_{\epsilon}^{(k)}$
such that $\phi(v) = (Q_{*,1}, \ldots, Q_{*,d-k})$, and by  \eqref{def-mapping Pi} we can construct $\{ \Gamma_{*,1}(v), \ldots,
 \Gamma_{*,d}(v) \}$ as an orthonormal basis of eigenvectors of $\tilde{x}$.  Therefore  we have the following representation 
\begin{align}\label{e:hatx}
	\tilde{x} = \Gamma(v) \Delta (u) \Gamma(v)^*=\bar G(u,v).
\end{align}

Now we choose $V = {\mathbf S}(d;k) \cap \mathfrak B_{\delta_0}(x_0)$ and $U = \bar G^{-1}(V)$.   Then $\bar G|_U$ is 
surjective. The continuity of the mapping $\bar G$ implies that $U$ is open in $\bR^{d-k+1}\times  I_{\epsilon}^{(k)}$. To 
show $\bar G|_U$ is a homeomorphism, it suffices to show the following conditions are satisfied:
\begin{enumerate}
	\item[(a1)] $\bar G|_U$ is injective;
	\item[(a2)] $\bar G^{-1}$ is continuous over $V$.
\end{enumerate}

%First we show that $\bar G: U \to  \cS(d;k) \cap \mathfrak B_{\delta_0}(x_0) $ is injective, for which it suffices to prove that, by the definition of the function $\bar G$, for any fixed $x\in \cS(d;k) \cap \mathfrak B_{\delta_0}(x_0)$,  there exists a {\it unique} $Q\in \mathbf O(d)$ such that $\tilde x$ has the decomposition $\tilde x=QEQ^*$ with $(Q_{*,1}, \dots, Q_{*, d-k})\in\mathbf O(d;k)\cap \mathfrak B_{\gamma_0}(D)$. 
%
%
%
%
%
%Note that  for any fixed $x\in \cS(d;k) \cap \mathfrak B_{\delta_0}(x_0)$  which has the decomposition $\tilde x=QE Q^*$,  its eigenvalues satisfy $u_1<\dots <u_{d-k+1}=\dots=u_d$ with $u_j=E_{jj}$ for $j=1, \dots, d$.  Furthermore, the matrix $(Q_{*,1}, \dots, Q_{*,d-k})$, which consists of the unit eigenvectors associated with the least $d-k$ eigenvalues,  is uniquely determined in $\mathfrak B_{\delta}(A)$ (we assume that $\delta$ is sufficiently small, say, $\delta<2$). By the definition of function $\Gamma$, there exists a unique $v\in I_\epsilon^{(k)}$ such that $\varphi_{*,j}(v)=\Gamma_{*,j}(v)=Q_{*,j}, j=1,\dots, d-k$.  Thus,  for the fixed $x\in \cS(d;k) \cap \mathfrak B_{\delta_0}(x_0)$,  only $(u, v)\in \bR^{d-k+1}\times   I_\epsilon^{(k)}$ with $u=(u_1,\dots, u_{d-k+1})$ satisfies \eqref{e:hatx}. Thus, $\bar G|_U$ is injective. Consequently, $\bar G^{-1}$ exists.  
%

First we show that $\bar G: U \to  {\mathbf S}(d;k) \cap \mathfrak B_{\delta_0}(x_0) $ is injective. Suppose that  for  
 $x\in {\mathbf S}(d;k) \cap \mathfrak B_{\delta_0}(x_0)$, it has the following spectral decompositions,
\begin{align*}
	\tilde{x} = \Gamma(v) \Delta(u) \Gamma(v)^* = \Gamma(v') \Delta(u') \Gamma(v')^*.
\end{align*}
We aim to show that $(u,v)=(u',v').$

Denote $u' = (u'_1, \ldots, u'_{d-k+1})$.   If $u_{i} \neq u'_i$ for some $1 \le i \le d-k$, then the corresponding unit 
eigenvectors $\Gamma_{*,i}(v)=\phi_{*,i}(v)$ and $\Gamma_{*,i}(v')=\phi_{*,i}(v')$ belong to different eigenspaces 
and hence are orthogonal, therefore, $\left\| \phi_{*,i}(v) - \phi_{*,i}(v') \right\|=\sqrt 2$. However, recall   
$\phi: I_\epsilon^{(k)} \to \mathbf O(d;k)\cap \mathfrak B_{\gamma}(A)$ with $\gamma<\sqrt 2/2$, and this implies that for $1\le i\le d-k$,
\begin{align*}
	\left\| \phi_{*,i}(v) - \phi_{*,i}(v') \right\|
	\le \left\| \phi_{*,i}(v) - A_{*,i} \right\| + \left\| \phi_{*,i}(v') - A_{*,i} \right\|
	< 2\gamma<\sqrt 2. 
\end{align*}
This is  a contradiction, and hence $u_i=u'_i$ for $1\le i\le d-k$, which further implies that  $u = u'$ noting that the set of $d-k+1$ eigenvalues are uniquely determined by $\tilde x$. 

Thus, the two unit vectors $\phi_{*,i}(v)$ and $\phi_{*,i}(v')$ belong to the same $1$-dimensional eigenspace $\cE_{u_i}^{\tilde x}$ for $1 \le i \le d-k$, and hence $\phi_{*,i}(v')\in\{\phi_{*,i}(v), -\phi_{*,i}(v)\}$. This implies that $\phi_{*,i}(v)=\phi_{*,i}(v')$ for $1\le i\le d-k$, i.e., $\phi(v)=\phi(v')$, noting that $\|\phi_{*,i}(v)-\phi_{*,i}(v')\|<\sqrt 2$. Therefore,  $v=v'$ since $\phi$ is a diffeomorphism.

 Now we show that the condition (a2) is satisfied.  Consider any 
sequence $\{x_n\}_{n \in \bN} \subseteq V$, such that $\lim_{n \rightarrow \infty} x_n = x \in V$. Let $(u_n, v_n) 
=\bar G^{-1}(x_n) \in \bR^{d-k+1}\times  I_{\epsilon}^{(k)}$ and $(u, v) = \bar G^{-1}(x) \in \bR^{d-k+1}\times I_{\epsilon}^{(k)}$, and thus
\begin{align*}
	x_n = \Gamma(v_n) \Delta(u_n) \Gamma(v_n)^*, \quad
	x = \Gamma(v) \Delta (u) \Gamma(v)^*.
\end{align*}
By the continuity of the functions $E_1, \ldots, E_d$ which map the matrices to their eigenvalues in ascending order, 
it is clear that
\begin{align} \label{eq-continuity of beta}
	\lim_{n \rightarrow \infty} u_n = u.
\end{align}

By the definition of $V$ and $\Gamma$, $\phi(u_n) \in {\mathbf O}(d;k) \cap \mathfrak B_{\gamma_0}(A)$. 
Recalling that  $\phi: I_\epsilon^{(k)} \to {\mathbf O}(d;k)\cap \mathfrak B_{\gamma}(A)$ is a diffeomorphism 
and that $0<\gamma_0<\gamma$, the set $K = \phi^{-1}({\mathbf O}(d;k) \cap \overline{\mathfrak B_{\gamma_0}(A)}) 
\subseteq I_\epsilon^{(k)}$ is compact, which implies the sequential compactness of the sequence $\{v_n\}_{n \in \bN}$. 
Let $v'$ be a limit point of the sequence, then there exists a subsequence $\{v_{m_n}\}_{n \in \bN}$, such that 
$\lim_{n \rightarrow \infty} v_{m_n} = v'$. By the continuity of the mappings $\Gamma$ and $\Delta$, we have
\begin{align*}
	x = \lim_{n \rightarrow \infty} x_{m_n}
	= \lim_{n \rightarrow \infty} \Gamma(v_{m_n}) \Delta(u_{m_n}) \Gamma(v_{m_n})^*
	= \Gamma(v') \Delta(u) \Gamma(v')^*.
\end{align*}
Hence, we have, for $i=1, \dots, d,$
$	\Gamma_{*,i}(v') \in \cE_{\Delta_{i,i}(u)}^{\tilde{x}}. $ Since $\Delta (u) \in \mathfrak B_{r_0}(D)$, we have 
$\Delta_{1,1}(u) < \cdots < \Delta_{d-k+1,d-k+1}(u)$, which implies that the eigenspace $\cE_{\Delta_{i,i}(u)}^{\tilde{y}}$ 
is $1$-dimensional for $1 \le i \le d-k$. Then similarly to the proof  of (a1), using the fact  that both $\phi(v)$ and 
$\phi(v')$ belong to  $\mathfrak B_{\gamma}(A)$ with $\gamma<\sqrt 2/2$, one can show that $\phi(v)=\phi(v')$, 
and hence $v=v'$ noting that $\phi$ is a diffeomorphism.  This implies 
\begin{equation}\label{eq-continuity of alpha}
\lim_{n\to\infty}  v_n =v,
\end{equation}
 since an arbitrary subsequence of $\{v_n, n\in\bN\}$ has a subsequence which converges to the common limit $v$.  

By \eqref{eq-continuity of beta} and \eqref{eq-continuity of alpha}, we obtain the continuity of the map $\bar G^{-1}$ 
and (a2) is proved.  The proof is concluded. 
\end{proof}

Now we are ready to prove Theorem \ref{Thm-hitting prob-real}.
\begin{proof}[Proof of Theorem \ref{Thm-hitting prob-real}]
We first prove Part (i). By Lemma \ref{Lem:dim-S}, there exists a smooth map 
$$
G: \bR^{\frac12[d(d+1) - (k+2)(k-1)]} \rightarrow {\mathbf S}(d)
$$ 
such that ${\mathbf S}(d;k)\subseteq \mathrm{Im}(G)$. As a consequence, the Hausdorff dimension of 
${\mathbf S}(d;k)$ is at most $\frac12[d(d+1) - (k+2)(k-1)]$. 

Notice that, recalling that $\beta=1$, 
\begin{equation}\label{Eq:Hit-up}
\begin{split}
	& \bP \left( \lambda_{i_1}^{\beta}(t) = \cdots = \lambda_{i_k}^{\beta}(t) \mathrm{\ for \ some \ } t \in I \ \mathrm{and} \ 
	1 \le i_1 < \cdots < i_k \le d \right) \\
	&= \bP \left( Y^{\beta}(t) \in {\mathbf S}(d;k) \mathrm{\ for \ some \ } t \in I \right) \\
	&= \bP \left( X^{\beta}(t) \in \left( {\mathbf S}(d;k)  - A^{\beta} \right) \mathrm{\ for \ some \ } t \in I \right) \\
	&\le \bP \left( X^{\beta}(t) \in \left( \mathrm{Im}(G) - A^{\beta} \right) \mathrm{\ for \ some \ } t \in I \right) \\
	&= \bP \left( X^{\beta}(I) \cap \left( \mathrm{Im}(G) - A^{\beta} \right) \not= \emptyset \right).
\end{split}
\end{equation}
By applying Lemma \ref{Lemma-Xiao-Hitting prob} to the last term in (\ref{Eq:Hit-up}), we see that 
\begin{align*}
	& \bP \left( \lambda_{i_1}^{\beta}(t) = \cdots = \lambda_{i_k}^{\beta}(t) \mathrm{\ for \ some \ } t \in I \ 
	\mathrm{and} \ 1 \le i_1 < \cdots < i_k \le d \right) \\
	%&= \bP \left( Y^{\beta}(t) \in {\mathbf S}(d;k) \mathrm{\ for \ some \ } t \in I \right) \\
	%&= \bP \left( X^{\beta}(t) \in \left( {\mathbf S}(d;k)  - A^{\beta} \right) \mathrm{\ for \ some \ } t \in I \right) \\
	%&\le \bP \left( X^{\beta}(t) \in \left( \mathrm{Im}(F) - A^{\beta} \right) \mathrm{\ for \ some \ } t \in I \right) \\
	%&= \bP \left( X^{\beta}(I) \cap \left( \mathrm{Im}(F) - A^{\beta} \right) \not= \emptyset \right) \\
	&\le c_6 {\mathbf H}_{d(d+1)/2-Q} \left( \mathrm{Im}(G) - A^{\beta} \right) \\
	&= c_6 {\mathbf H}_{d(d+1)/2-Q} \left( \mathrm{Im}(G) \right)\\
	& = 0,
\end{align*}
where $Q = \sum_{j=1}^N \frac{1}{H_j}$. 
In the above, the first equality follows from the translation invariance of Hausdorff measure and the second equality 
follows from the fact that  $d(d+1)/2-Q > \frac12[d(d+1) - (k+2)(k-1)]$. This proves (i) in Theorem \ref{Thm-hitting prob-real}.

Next, we prove Part (ii). We choose $x_0 \in {\mathbf S}(d;k)$ satisfying $\spec(\tilde{x}_0) = d-k+1$. By Lemma \ref{Lem:dim-bound}, there exists 
$\delta_0 > 0$, such that ${\mathbf S}(d;k)  \cap \mathfrak B_{\delta_0}(x_0)$ is a $\frac12[(d(d+1) - (k+2)(k-1)]$-dimensional manifold. Then, 
similarly to (\ref{Eq:Hit-up}), we have
\begin{align*}
	& \bP \left( \lambda_{i_1}^{\beta}(t) = \cdots = \lambda_{i_k}^{\beta}(t) \mathrm{\ for \ some \ } t \in I \ \mathrm{and} 
	\ 1 \le i_1 < \cdots < i_k \le d \right) \\
	&= \bP \left( X^{\beta}(t) \in \left( {\mathbf S}(d;k) - A^{\beta} \right) \mathrm{\ for \ some \ } t \in I \right) \\
	&\ge \bP \left( X^{\beta}(t) \in \left( {\mathbf S}(d;k)  \cap \mathfrak B_{\delta_0}(x_0) - A^{\beta} \right) \mathrm{\ for \ some \ } t \in I \right) \\
	&= \bP \left( X^{\beta}(I) \cap \left( {\mathbf S}(d;k)  \cap \mathfrak B_{\delta_0}(x_0) - A^{\beta} \right) \not= \emptyset \right) \\
	&\ge c_5 \cC_{d(d+1)/2-Q} ({\mathbf S}(d;k)  \cap \mathfrak B_{\delta_0}(x_0) - A^{\beta}) \\
	&= c_5 \cC_{d(d+1)/2-Q} ({\mathbf S}(d;k) \cap \mathfrak B_{\delta_0}(x_0)) > 0.
\end{align*}
In the above, we have used the lower bound on hitting probability in Lemma \ref{Lemma-Xiao-Hitting prob} and the last step 
follows from the fact that   $d(d+1)/2-Q < d(d+1)/2 - (k+2)(k-1)/2$.

Finally, we prove Part (iii) by applying Lemma \ref{Lem:dim}. Notice that
\[
 {\cal C}_k^\beta = (Y^{\beta})^{-1}( {\mathbf S}(d;k) ) \subseteq (X^{\beta})^{-1} \big(\mathrm{Im}(G) - A^{\beta}\big),
 \]
where $(Y^{\beta})^{-1}(B) = \{t: \, Y^{\beta}(t) \in B\}$ denotes the inverse image of $B$ under the mapping $Y^{\beta}$.
By applying Part ($a$) of Lemma \ref{Lem:dim} with $B = \mathrm{Im}(G) - A^{\beta}$ we have 
\begin{equation}\label{Eq:dimC1}
 \dimh  {\cal C}_k^\beta \le  \min_{1 \le \ell \le N}\bigg\{\sum_{j=1}^{\ell } \frac{H_\ell}{H_j} + N-\ell - H_\ell \frac{(k+2)(k-1)} 2\bigg\}  
 \ \ \hbox{ a.s.}
 \end{equation}
On the other hand, for any $x_0 \in {\mathbf S}(d;k) $ with with $|\spec(\tilde{x}_0)| = d-k+1$, let  $\mathfrak B_{\delta_0}(x_0)$ be the open ball in 
Lemma \ref{Lem:dim-bound}. Since 
$$
{\cal C}_k^\beta \supseteq (Y^{\beta})^{-1}( {\mathbf S}(d;k) \cap \mathfrak B_{\delta_0}(x_0)) = (X^{\beta})^{-1} \big({\mathbf S}(d;k) \cap\mathfrak B_{\delta_0}(x_0) - A^{\beta}\big)
$$
and the Lebesgue measure on ${\mathbf S}(d;k) \cap \mathfrak B_{\delta_0}(x_0) - A^{\beta}$ satisfies condition (\ref{Eq:mu0}).
It follows from ($b$) of Lemma \ref{Lem:dim} that with positive probability,
\begin{equation}\label{Eq:dimC2}
 \dimh  {\cal C}_k^\beta \ge  \min_{1 \le \ell \le N}\bigg\{\sum_{j=1}^{\ell } \frac{H_\ell}{H_j} + N-\ell - H_\ell \frac{(k+2)(k-1)} 2\bigg\}.
 \end{equation}
Thus the first equality in (\ref{Eq:dimC}) follows from \eqref{Eq:dimC1} and \eqref{Eq:dimC2}. The second equality in (\ref{Eq:dimC}) is 
elementary and can be verified directly. This finishes the proof of Theorem \ref{Thm-hitting prob-real}.
\end{proof}

\section{The complex case: proof of Theorem \ref{Thm-hitting prob-complex}}

In this section, we consider the matrix \eqref{def-entries} for the case $\beta = 2$ and we 
develop our results following the idea of \cite{Jaramillo2018}.

\begin{lemma} \label{Lemma-Prop 4.1}
For every $A \in {\mathbf U}(d;k)$, there exists a constant $\delta > 0$, such that the set
\begin{align*}
\Upsilon_{\delta}^A = \left\{ B \in {\mathbf U}(d;k) \cap \mathfrak B_{\delta}(A): \langle A_{*,j}, 
B_{*,j} \rangle =|\langle A_{*,j}, B_{*,j} \rangle|, ~ 1 \le j \le d-k \right\}
\end{align*}
is a $(d^2 - d - k^2 + k)$-dimensional submanifold of ${\mathbf U}(d;k) \cap \mathfrak B_{\delta}(A)$.
\end{lemma}

\begin{proof}
Consider the manifold $\bT_{d-k} = \{e^{\iota \theta}: \theta \in [-\pi/2, \pi/2)\}^{d-k} \subseteq \bC^{d-k}$ 
and define a smooth mapping $f: {\mathbf U}(d;k) \cap \mathfrak B_{\delta}(A) \rightarrow \bT_{d-k}$ by
\begin{align*}
	f(B) = \left( \dfrac{\langle A_{*,1}, B_{*,1} \rangle}{|\langle A_{*,1}, B_{*,1} \rangle|}, \ldots, 
	\dfrac{\langle A_{*,d-k}, B_{*,d-k} \rangle}{|\langle A_{*,d-k}, B_{*,d-k} \rangle|} \right).
\end{align*}
Noting that $\|A_{*,j}\| = 1$ for $1 \le j \le d-k$, the mapping $f$ is well-defined for sufficiently small $\delta> 0$.

We denote $\mathbf{1} := f(A) = (1, \ldots, 1)$. Note that $\Upsilon^A_{\delta}=f^{-1}(\mathbf 1)$ and the tangent space $T_{\mathbf 1}\mathbb T_{d-k}$ of $\bT_{d-k}$ at 
$\mathbf{1}$ is $\iota \bR^{d-k}$. Suppose $B\in \Upsilon^A_{\delta}$, i.e.,  $B \in {\mathbf U}(d;k) \cap\mathfrak B_{\delta}(A)$ such that $f(B) =
 \mathbf{1}$ and let  $v = (v_1, \ldots, v_{d-k}) \in \iota \bR^{d-k}$ be an arbitrary fixed element in $T_{\mathbf 1}\mathbb T_{d-k}$. Then there exists $\epsilon>0$ such that the curve $\theta: (-\epsilon, \epsilon) \rightarrow \Upsilon_{\delta}^A$ 
 on ${\mathbf U}(d;k) \cap
\mathfrak B_{\delta}(A)$ given by
\begin{align*}
	\theta_{*,j}(t) = \exp( v_j t) B_{*,j}
\end{align*}
satisfies $\theta(0) = B$. Moreover, $D_Bf(\dot \theta(0)) =\frac{d}{dt} f(\theta(t))|_{t=0}=v$. Thus, noting that $\dot \theta(0)\in T_B \Upsilon^{A}_\delta$ and $v$ is an arbitrary vector in $T_{\mathbf 1} \mathbb T_{d-k}$,  we have shown that $D_Bf : T_B \Upsilon^A_{\delta}\to T_{\mathbf 1} \mathbb T_{d-k}$ 
is surjective, which implies that $\mathbf{1}$ is a regular value. Therefore, by Lemma \ref{Lemma-inverse of 
regular value}, $\Upsilon_{\delta}^A=f^{-1}(\mathbf 1)$ is a submanifold of  ${\mathbf U}(d;k) \cap \mathfrak B_{\delta}(A)$ of dimension $(d^2 - k^2) - (d-k)$.
\end{proof}

The following lemma is the complex version of Lemma \ref{Lem:dim-S}, which shows that $d^2-k^2+1$ is an upper bound for the dimension of $\mathbf H(d;k)$.  

\begin{lemma} \label{Lemma-Prop 4.6}
(i) Fix $k\in\{1, \dots, d-1\}$. Let $\Delta: \bR^{d-k+1} \rightarrow {\mathbf D}(d)$ be the function that maps the vector $u = (u_1, \ldots, u_{d-k+1}) 
\in \bR^{d-k+1}$ to the diagonal matrix $\Delta(u) = \{\Delta_{i,j}(u):1 \le i, j \le d\}$ given by \eqref{def-Lamda mapping}. Then there 
exists a compactly supported smooth function $\widehat{\Gamma}: \bR^{d^2 - d - k^2 + k} \rightarrow \bC^{d \times d}$
such that the mapping $\widehat{G}: \bR^{d-k+1} \times \bR^{d^2 - d - k^2 + k} \rightarrow {\mathbf H}(d)$ given by
\begin{align} \label{def-F tilde function}
	\widehat{G}(u, v) = \widehat{\Gamma}(v) \Delta(u) \widehat{\Gamma}(v)^*, ~~  (u,v)   \in 
	\bR^{d-k+1}\times \bR^{d^2 - d - k^2 + k}
\end{align}
satisfies
\begin{align} \label{eq-Lemma-Prop 4.6}
	{\mathbf H}(d;k)\subseteq \left\{ x \in \bR^{d^2}: \ \tilde{x} \in \mathrm{Im}(\widehat{G}) \right\},
\end{align}
where $\mathrm{Im}(\widehat{G})$ is the image of the mapping $\widehat{G}$.

(ii) If $k=d$, the dimension of ${\mathbf H}(d;d)$ is 1.

\end{lemma}

\begin{proof}
The proof is similar to the proof of Lemma \ref{Lem:dim-S}. We sketch the proof of (i) below. 

For $\epsilon > 0$, we denote $\widehat{I}_{\epsilon}^{(k)} = (-\epsilon, \epsilon)^{d^2 - d - k^2 + k}$. For each 
$A \in {\mathbf U}(d;k)$, we can choose $\hat A \in {\mathbf U}(d)$, such that $A_{*,j} = \hat A_{*,j}$ for $1 \le j \le d-k$. 
By Lemma \ref{Lemma-Prop 4.1}, the dimension of the smooth manifold $\Upsilon_{\delta}^A$ is $d^2 - d - k^2 + k$, 
where  $\delta$ is a sufficiently small positive constant. Hence, there exists $\epsilon > 0$ and  a smooth diffeomorphism 
$\widehat{\phi}: \widehat{I}_{\epsilon}^{(k)} \rightarrow \Upsilon_{\delta}^A$ with $\widehat{\phi}(0) = A$. As in the proof 
of Lemma \ref{Lem:dim-S}, the functions $\{\widehat{\varphi}_{k-1}, \ldots, \widehat{\varphi}_0\}$ 
given by
\begin{align*}
	\widehat{\varphi}_l(v) = \dfrac{\hat A_{*,d-l} - \sum_{i=l+1}^{k-1} \langle \widehat{\varphi}_i(v), 
	\hat A_{*,d-l} \rangle \widehat{\varphi}_i(v) - \sum_{j=1}^{d-k} \langle \widehat{\phi}_{*,j}(v), 
	\hat A_{*,d-l} \rangle \widehat{\phi}_{*,j}(v)}{\| \hat A_{*,d-l} - \sum_{i=l+1}^{k-1} \langle \widehat{\varphi}_i(v), 
	\hat A_{*,d-l} \rangle \widehat{\varphi}_i(v) - \sum_{j=1}^{d-k} \langle \widehat{\phi}_{*,j}(v), 
	\hat A_{*,d-l} \rangle \widehat{\phi}_{*,j}(v) \|}
\end{align*}
are smooth for $0 \le l \le k-1$. Let $\widehat{\Gamma}: \bR^{d^2 - d - k^2 +k} \rightarrow \bC^{d \times d}$ be 
a smooth mapping with compact support satisfying
\begin{align} 
	\widehat{\Gamma}_{*,j}(v) =
	\begin{cases}
		\widehat{\phi}_{*,j}(v), & 1 \le j \le d-k; \\
		\widehat{\varphi}_{d-j}(v), & d-k+1 \le j \le d,
	\end{cases}
\end{align}
for all $v \in \widehat{I}_{\epsilon}^{(k)}$.  Let $g^A: \mathbf U(d;k)\cap\mathfrak B_\delta(A) \to \Upsilon_\delta^A$ be a smooth mapping with the columns given by 
\[ g^A_{*,j}(B)=\frac{|\langle B_{*,j}, A_{*,j}\rangle |}{\langle B_{*,j}, A_{*,j}\rangle } B_{*,j}, ~~ 0\le j\le d-k.\]
Recall that the set
\begin{align*}
	\widehat{U}_{\widehat{\Gamma}}(A) = \left\{ \left( \widehat{\Gamma}_{*,1}(v), \ldots, \widehat{\Gamma}_{*,d-k}(v) \right) : v \in \widehat{I}_{\epsilon}^{(k)} \right\}
	= \widehat{\phi} (\widehat{I}_{\epsilon}^{(k)})=\Upsilon_\delta^A
\end{align*}
is open, and hence $(g^A)^{-1}(\Upsilon_\delta^A)= (g^A)^{-1}(\widehat U_{\widehat \Gamma}(A)):=\widehat V_{\widehat \Gamma}(A)$ is open in $\mathbf U(d;k)$.   Thus,  the collection $\{\widehat V_{\widehat \Gamma}(A): A \in {\mathbf U}(d;k)\}$ forms an open cover for the compact set ${\mathbf U}(d;k)$, and there exists   
a finite subcover for $\mathbf U(d;k)$, say, 
\begin{align*}
	{\mathbf U}(d;k) = \bigcup_{l=1}^L \widehat{V}_{\widehat \Gamma^{(l)}}(A_l).
\end{align*}
We  define mappings $\widehat{G}^{(l)}$ by $\widehat{G}^{(l)} (u, v) = \widehat{\Gamma}^{(l)} (v) \Delta(u) \widehat{\Gamma}^{(l)} (v)^*$ for $(u,v) \in \bR^{d-k+1}\times \widehat{I}_{\epsilon_l}^{(k)}$ for  $l=1, \dots, L$.

For $ x \in {\mathbf H}(d;k)$ with decomposition $\tilde{x} = Q \Delta Q^*$ with $\Delta \in {\mathbf D}(d)$ and $Q\in {\mathbf U}(d)$, we assume that $\Delta_{d-k+1,d-k+1} = \cdots = \Delta_{d,d}$ and $\Delta = \Delta(u)$ for some $u\in \bR^{d-k+1}$. For $(Q_{*,1}, \dots, Q_{*,d-k})\in \mathbf U(d;k)$, there exist $1 \le l_0 \le L$,  such that $(Q_{*,1}, \ldots, Q_{*,d-k}) \in \widehat V_{\widehat\Gamma^{(l_0)}}(A_{l_0})$, and hence there exists $v \in \widehat{I}_{\epsilon_{l_0}}^{(k)}$ such that $(Q_{*,1}, \ldots, Q_{*,d-k}) = (g^{A_{l_0}})^{-1} (\widehat \Gamma_{*,1}^{(l_0)} (v),\dots, \widehat \Gamma_{*,d-k}^{(l_0)} (v))$. Therefore,  $\widehat{\Gamma}_{*,1}^{(l_0)}(v), \ldots, \widehat{\Gamma}_{*,d-k}^{(l_0)}(v)$ are unit eigenvectors of $\tilde x$ associated with $\Delta_{1,1}(u),\dots, \Delta_{d-k, d-k}(u)$, and  $\widehat{\Gamma}_{*,d-k+1}^{(l_0)}(v), \ldots, \widehat{\Gamma}_{*,d}^{(l_0)}(v)$ are orthonormal eigenvector of $\tilde{x}$ associated with the eigenvalue $\Delta_{d-k+1,d-k+1}(u)$. Therefore, we have shown that  $\tilde{x} = \widehat{\Gamma}^{(l_0)} (v) \Delta(u) \widehat{\Gamma}^{(l_0)} (v)^*$.

Finally, the function  $\widehat G$ can be constructed via $\{\widehat G^{(l)}, l=1,\dots, L\}$ in the same way as the function  
$G$  in the proof of Lemma \ref{Lem:dim-S}.  The proof is concluded. 
\end{proof}

The following result is the complex version of Lemma \ref{Lem:approx-S}. The proof is similar to that of Lemma \ref{Lem:approx-S} 
and hence is omitted.

\begin{lemma} \label{Lemma-Prop 4.4} Fix $k\in\{1, \dots, d\}$. Let $A \in {\mathbf H}(d;k)$ be a Hermitian 
matrix with decomposition $A = P D P^*$ for some $P \in {\mathbf U}(d)$ and $D \in {\mathbf D}(d)$ 
such that $|\spec(A)| = d-k+1$. Then for every $\epsilon > 0$, there exists $\delta > 0$, such that for all 
$B \in {\mathbf H}(d;k)$ satisfying
\begin{align*}
	\max_{1 \le i, j \le d} |A_{i,j} - B_{i,j}| < \delta,
\end{align*}
we have $|\spec(B)|=d-k+1$, and there exists a spectral decomposition of the form $B = Q F Q^*$, 
where $Q \in {\mathbf U}(d)$ and $F \in {\mathbf D}(d)$ satisfy
\begin{align*}
\max_{1 \le i \le d} |D_{i,i} - F_{i,i}| < \epsilon, \quad \max_{1 \le i, j \le d} |Q_{i,j} - P_{i,j}| < \epsilon.	
\end{align*}
\end{lemma}

The following lemma is the complex version of Lemma \ref{Lem:dim-bound}, which indicates that $d^2-k^2+1$ is an optimal upper 
bound for the dimension of $\mathbf H(d;k)$. 

\begin{lemma} \label{Lemma-Prop 4.8} Let $k\in\{1, \dots, d\}$.  For $x_0 \in {\mathbf H}(d;k)$ with $|\spec(\tilde{x}_0)| = d-k+1$, 
there exists $\delta_0 > 0$, such that ${\mathbf H}(d;k) \cap \mathfrak B_{\delta_0}(x_0)$ is a $(d^2 - k^2 + 1)$-dimensional 
manifold with positive $(d^2 - k^2 + 1)$-dimensional Lebesgue measure.
\end{lemma}

\begin{proof}
The proof is similar to the proof of Lemma \ref{Lem:dim-bound}, and we sketch it for $k\in\{1, \dots, d-1\}$.

Suppose $\tilde{x}_0 = P D P^*$, where $P \in {\mathbf U}(d)$, $D \in {\mathbf D}(d)$ with $D_{1,1} < \cdots < D_{d-k+1, d-k+1} 
= \cdots D_{d,d}$.  Recall  $\widehat{I}_{\epsilon}^{(k)} = (-\epsilon, \epsilon)^{d^2 - d - k^2 + k}$.  Let $A \in {\mathbf U}(d;k)$ be 
the matrix obtained from $P$ by deleting last $k$ columns. Then for some sufficiently small $\epsilon, \gamma > 0$, we have a 
diffeomorphism $\widehat{\phi}: \widehat{I}_{\epsilon}^{(k)} \rightarrow \Upsilon_{\gamma}^A$ such that $\widehat{\phi}(0) = A$. 
As in the proof of Lemma \ref{Lem:dim-bound}, we can construct a smooth map $\widehat{\Gamma}: \widehat{I}_{\epsilon}^{(k)} 
\rightarrow {\mathbf U}(d)$, such that $\widehat{\Gamma}_{*,j} (v) = \widehat{\phi}_{*,j} (v)$ for all $v \in \widehat{I}_{\epsilon}^{(k)}$ 
and $1 \le j \le d-k$.

Next, we define the map $\bar{G}:\bR^{d-k+1} \times \widehat{I}_{\epsilon}^{(k)} \rightarrow {\mathbf H}(d;k)$ by $\bar{G}(u,v) 
= \widehat{\Gamma}(v) \Delta(u) \widehat{\Gamma}(v)^*$. Then by Lemma \ref{Lemma-Prop 4.4}, for some fixed $\gamma_0\in
(0, \gamma)$, there exists $\delta_0 > 0$ such that for all $x \in {\mathbf H}(d;k) \cap \mathfrak B_{\delta_0}(x_0)$, there exist 
$Q \in {\mathbf U}(d)$ and $E\in {\mathbf D}(d)$, such that $\tilde{x} = Q E Q^*$ with
\begin{align*}
Q \in {\mathbf U}(d) \cap \mathfrak B_{\gamma_0}(P), \quad
	E \in {\mathbf D}(d) \cap \mathfrak B_{r_0}(D),
\end{align*}
where $r_0$ is given by
\begin{align*}
	r_0 = \dfrac{1}{2} \min_{\substack{\mu , \lambda \in \spec(\tilde{x}_0)\\ \mu\neq \lambda}} |\mu -\lambda|.
\end{align*}

For $x = Q \Delta Q^*$, if we multiply the $j$-th column $Q_{*,j}$ of $Q$ by $\frac{|\langle Q_{*,j}, A_{*,j}\rangle|}{\langle Q_{*,j}, A_{*,j}\rangle}$ 
for $1 \le j \le d-k$, the decomposition still holds, which is again denoted by $x = Q \Delta Q^*$. Now $(Q_{*,1}, \ldots, Q_{*,d-k}) \in \Upsilon_{{\gamma_0}}^A\subset \Upsilon_{{\gamma}}^A$. 
Hence, there exists $v \in \widehat{I}_{\epsilon}^{(k)}$ such that $\widehat{\phi} (v) = (Q_{*,1}, \ldots, Q_{*,d-k})$. 
Proceeding as in the proof of Lemma \ref{Lem:dim-bound}, we obtain the decomposition 
\begin{align} \label{eq-3.4}
	\tilde{x} = \widehat{\Gamma}(v) \Delta(u) \widehat{\Gamma} (v)^*=\bar G(u,v),
\end{align}
for some $u \in \bR^{d-k+1}$.

 Now we choose $V = {\mathbf H}(d;k) \cap\mathfrak B_{\delta_0}(x_0)$ and $U = \bar{G}^{-1}(\widetilde{V})$.  Then $\bar G|_U$ is surjective and $U$ is open in $\bR^{d-k+1}\times \widehat{I}_{\epsilon}^{(k)}$. As in the proof of Lemma \ref{Lem:dim-bound}, to prove the desired result, it suffices to show the following conditions hold:
\begin{enumerate}
	\item[(b1)]  $\bar G|_U$ is injective;
	\item[(b2)] $\bar{G}^{-1}$ is continuous over $V$.
\end{enumerate}

First we show that $\bar{G}: U \to  {\mathbf H}(d;k) \cap \mathfrak B_{\delta_0}(x_0) $ is injective. Suppose that for $x\in \mathbf H(d;k)\cap \mathfrak B_{\delta_0}(x_0)$, it has the following spectral decompositions, 
\begin{align*}
\tilde{x} = \widehat{\Gamma}(v) \Delta(u) \widehat{\Gamma}(v)^*=\widehat{\Gamma}(v') \Delta(u') \widehat{\Gamma}(v')^*.
\end{align*}
We shall show that $(u,v)=(u',v').$

 Similar to showing (a1) in the proof of Lemma \ref{Lem:dim-bound}, one can deduce that $u = u'$ first. Thus, the two unit vectors $\widehat{\Gamma}_{*,i}(v)$ and $\widehat{\Gamma}_{*,i}(v')$ belong to the same $1$-dimensional complex eigenspace $\cE_{\Delta_{i,i}(u)}^{\tilde x}$ for $1 \le i \le d-k$. Therefore, there exists $\theta \in \bC$ with $|\theta| = 1$ such that $\widehat\Gamma_{*,i}(v') = \theta \widehat\Gamma_{*,i}(v)$ for $1\le i\le d-k$.
 
On the other hand,  the vector  $\widehat \phi(w)=(\widehat{\Gamma}_{*,1} (w), \ldots, 
\widehat{\Gamma}_{*,d-k} (w))$ is in $\Upsilon_{\gamma}^A$ for all $w \in \widehat{I}_{\epsilon}^{(k)}$,  and this implies that $\theta$ is a non-negative real number. As a consequence, we have $\theta = 1$ and $\widehat{\phi} (v) = \widehat{\phi} (v')$. Therefore $v = v'$ since $\widehat{\phi}$ is a diffeomorphism.

 Now we show that  (b2) is satisfied, following a similar way as for (a2) in the proof of Lemma \ref{Lem:dim-bound}.  Consider any 
sequence $\{x_n\}_{n \in \bN} \subseteq V$, such that $\lim_{n \rightarrow \infty} x_n = x \in V$. Let $(u_n, v_n) 
=\bar G^{-1}(x_n) \in \bR^{d-k+1}\times \widehat I_{\epsilon}^{(k)}$ and $(u, v) = \bar G^{-1}(x) \in \bR^{d-k+1}\times I_{\epsilon}^{(k)}$, and thus
\begin{align*}
	x_n = \widehat\Gamma(v_n) \Delta(u_n) \widehat\Gamma(v_n)^*, \quad
	x =\widehat \Gamma(v) \Delta (u)\widehat \Gamma(v)^*.
\end{align*}

Using similar arguments in the proof of Lemma \ref{Lem:dim-bound}, one can show that $\lim_{n \rightarrow \infty} u_n = u$ 
and that the sequence $\{ v_n \}_{n\in \mathbb{N}}$ is compact.  Let $v'$ be a limit point 
of the sequence, then there exists a subsequence $\{v_{m_n}\}_{n \in \bN}$, such that $\lim_{n \rightarrow \infty} v_{m_n} 
= v'$. By the continuity of the mappings $\widehat\Gamma$ and $\Delta$, we have
\begin{align*}
	x = \lim_{n \rightarrow \infty} x_{m_n}
	= \lim_{n \rightarrow \infty} \widehat \Gamma(v_{m_n}) \Delta(u_{m_n}) \widehat\Gamma(v_{m_n})^*
	= \widehat\Gamma(v') \Delta(u) \widehat\Gamma(v')^*.
\end{align*}
 Since $\Delta (u) \in \mathfrak B_{r_0}(D)$, we have $\Delta_{1,1}(u) < \cdots < \Delta_{d-k+1,d-k+1}(u)$, and hence
 the eigenspace $\cE_{\Delta_{j,j}(u)}^{\tilde{x}}$ is $1$-dimensional for $1 \le j \le d-k$.  Thus, there exists $\theta \in \bC$ 
 with $|\theta| = 1$, such that 
\begin{align*}
	\widehat{\Gamma}_{*,j}(v') = \theta \widehat{\Gamma}_{*,j}(v), \quad \forall 1 \le j \le d-k.
\end{align*}
Thus, similar to the proof of (b1), we have $\theta = 1$ and  hence $\widehat{\phi} (v') = \widehat{\phi} (v)$, 
which implies that $v=v'$ recalling that  $\widehat{\phi}$ is a diffeomorphism.  This implies   
that $\lim_{n\to\infty}v_n=v$. Thus $\lim_{n\to\infty}(u_n,v_n)=(u,v)$ and (b2) is proved. 
This finishes the proof of Lemma \ref{Lemma-Prop 4.8}.  
\end{proof}

Now we are ready to prove Theorem \ref{Thm-hitting prob-complex}. 
 
\begin{proof}[Proof of Theorem \ref{Thm-hitting prob-complex}]
The proof is similar to that of Theorem \ref{Thm-hitting prob-real}, which is sketched below.

We first prove Part (i). By Lemma \ref{Lemma-Prop 4.6}, there exists a smooth map $\widehat{G}: \bR^{d-k+1}\times \bR^{d^2 - d - k^2 + k} 
 \rightarrow {\mathbf H}(d)$, such that ${\mathbf H}(d;k) \subseteq \mathrm{Im}(\widehat{G})$. The Hausdorff 
dimension of the image set $\mathrm{Im}(\widehat{G})$ is at most $d^2 - k^2 + 1$. Thus, by Lemma
\ref{Lemma-Xiao-Hitting prob}, denoting $Q=\sum_{j=1}^N \frac1{H_j}$,  we derive thart if  $d^2-Q > d^2 - k^2 + 1$,
\begin{align*}
	&\bP \left( \lambda_{i_1}^{\beta}(t) = \cdots = \lambda_{i_k}^{\beta}(t) \mathrm{\ for \ some \ }
	 t \in I \ \mathrm{and} \ 1 \le i_1 < \cdots < i_k \le d \right) \\
	&= \bP \left( Y^{\beta}(t) \in {\mathbf S}(d;k) \mathrm{\ for \ some \ } t \in I \right) \\
	&\le \bP \left( X^{\beta}(I) \cap \big( \mathrm{Im}(\widehat G) - A^{\beta} \big) \not= \emptyset \right) \\
	&\le c_6 {\mathbf H}_{d^2-Q} \big( \mathrm{Im}(\widehat G) - A^{\beta} \big) = 0.
\end{align*}

Next, we prove (ii). We choose $x_0 \in {\mathbf H}(d;k)$ satisfying $\spec(\tilde{x}_0) = d-k+1$. By Lemma \ref{Lemma-Prop 4.8}, 
there exists $\delta_0 > 0$, such that ${\mathbf H}(d;k) \cap \mathfrak B_{\delta_0}(x_0)$ is an $(d^2 - k^2 + 1)$-dimensional manifold. 
Thus, by Lemma \ref{Lemma-Xiao-Hitting prob}, when $d^2-Q < d^2 - k^2 + 1$, we have
\begin{align*}
	&\bP \left( \lambda_{i_1}^{\beta}(t) = \cdots = \lambda_{i_k}^{\beta}(t) \mathrm{\ for \ some \ } t \in I \ 
	\mathrm{and} \ 1 \le i_1 < \cdots < i_k \le d \right) \\
	&= \bP \left( X^{\beta}(t) \in \left( {\mathbf S}(d;k) - A^{\beta} \right) \mathrm{\ for \ some \ } t \in I \right) \\
	&\ge \bP \left( X^{\beta}(I) \in \left( {\mathbf S}(d;k) \cap\mathfrak B_{\delta_0}(x_0) - A^{\beta} \right) \not= \emptyset \right) \\
	&\ge c_5 \cC_{d^2-Q} ({\mathbf S}(d;k)\cap \mathfrak B_{\delta_0}(x_0) - A^{\beta}) > 0.
\end{align*}

The proof of (iii) is similar to that in the proof of Theorem  \ref{Thm-hitting prob-real} and is omitted. This finishes the proof.
\end{proof}

\section{Appendix}\label{sec:Appendix}

\subsection{Manifold}\label{sec:manifold}

We collect some materials on manifold which are used in the proofs. The reader is referred to \cite{Tu2011} for more details. 

Let $M$ be a smooth submanifold of $\bR^n$, its tangent plane $T_x M$ at $x\in M$ is defined as the set of vectors 
of the form $\theta'(0)$, where $\theta:(-\epsilon, \epsilon)\to M$ for some $\epsilon>0$ is a smooth curve satisfying 
$\theta(0)=x$. Let $M,N$ be smooth manifolds and  $f:M\to N$ is a smooth mapping. The derivative of $f$ at $x\in M$, 
denoted by $D_xf$, is defined as the mapping $D_xf: T_xM\to T_{f(x)}N$ which maps $v=\theta'(0)\in T_xM$ to $D_xf(v)
:=\frac{d}{dt} f(\theta(t))|_{t=0}$.

\begin{definition} \label{Def-regular value}
Suppose $f : M \rightarrow N$ is a smooth map between smooth manifolds. A point $q \in N$ is called a regular value if $f$ is a 
submersion at each $p \in f^{-1}(q)$, i.e. $D_pf : T_pM \rightarrow T_qN$ is surjective.
\end{definition}

\begin{lemma} \label{Lemma-inverse of regular value}
If $q$ is a regular value of a smooth map $f : M \rightarrow N$, then $f^{-1}(q)$ is a submanifold of $M$ of dimension $\dim M - \dim N$.
\end{lemma}

\subsection{Hausdorff dimension and capacity}

We recall briefly the definitions of Hausdorff measure, Hausdorff dimension, and Bessel-Riesz capacity
that are used in this paper. For a systematic account on these and other fractal dimensions we refer to 
\cite{Falconer2014} or Mattila \cite{mattila1999}.

Let $n \in \bN$ be fixed and $q > 0$ be a constant.  For any subset $A \subseteq \bR^n$,  the $q$-dimensional 
Hausdorff measure of $A$ is defined by
\begin{equation*}
%\label{Eq:Hausdorff} 
\mathcal{H}_{q} (A) = \lim_{\varepsilon \to 0}\
\inf \bigg\{ \sum_i (2 r_i)^q:\, A \subseteq \bigcup_{i
=1}^{\infty} \mathfrak B_{r_i}(x_i), \  r_i < \varepsilon \bigg\}.
\end{equation*}
It is known that
$\mathcal{H}_{q} (\cdot)$  is a metric outer measure and every Borel set in
$\bR^n$ is  $\mathcal{H}_{q}$-measurable.  The Hausdorff dimension $\dimh (A)$ defined by
\begin{align*}
\dimh (A) = \inf \{q>0 \ | \ {\mathbf H}_q(A) = 0\}
	= \sup \{q>0 \ | \ {\mathbf H}_q(A) > 0\}.
\end{align*}

The Bessel-Riesz capacity of order $q$ of $A$ is defined by
\begin{align*}
	\cC_{q} (A)
	= \left( \inf_{\mu \in \mathcal{P}(A)}  \iint_{\bR^n \times \bR^n} f_q(\|x - y\|) \mu(dx) \mu(dy) \right)^{-1},
\end{align*}
where $\mathcal{P}(A)$ is the family of probability measures supported in $A$, and the function 
$f_q: \bR_+ \rightarrow \bR_+$ is given by
\begin{align*}
	f_q(r) =
	\begin{cases}
	r^{-q}, & q > 0; \\
	\ln \left( \dfrac{e}{r \wedge 1} \right), & q= 0; \\
	1, & q < 0.
	\end{cases}
\end{align*}

Hausdorff dimension and Bessel-Riesz capacity are related by the following Frostman theorem:
\begin{align*}
	\dimh (A) = \inf \{q>0 \ | \ \cC_q(A) = 0\}
	= \sup \{q>0 \ | \ \cC_q(A) > 0\}.
\end{align*}

\subsection{Hitting probabilities}

Finally, we collect some results from \cite{Xiao2009} on the hitting probabilities of Gaussian random fields that satisfy
conditions  $(H1)$ and $(H2)$. For further development see \cite{dalang2017polarity}.

Let $n \in \bN$, let $W = \{(W_1(t), \ldots, W_n(t)): t \in \bR_+^N\}$ be an $n$-dimensional Gaussian random field, 
whose entries are independent copies of $\{\xi(t): t \in \bR_+^N\}$.  Then we make use of the following lemmas.

\begin{lemma} \cite[Theorem 2.1]{Xiao2009} \label{Lemma-Xiao-Hitting prob}
Consider the interval $I$ of the form \eqref{def-interval}. Suppose that the assumptions (A1), (A2) hold. 
Let $B \subseteq \bR^n$ be a Borel set, then there exist positive constants $c_5, c_6$ that depend only on $I$, $G$, $H$, 
such that
\begin{align*}
	c_5 \cC_{n-Q} (B)
	\le \bP \left( W^{-1}(B) \cap I \not= \emptyset \right)
	\le c_6 {\mathbf H}_{n-Q} (B),
\end{align*}
where $Q = \sum_{j=1}^N \frac{1}{H_j}$. 
\end{lemma}
%{\blue Are we gonna use the following Lemma?}
When $W^{-1}(B) \cap I \not= \emptyset$, its Hausdorff dimension is studied in \cite{Xiao2009}. For proving Part (iii) in 
Theorems \ref{Thm-hitting prob-real}  and  \ref{Thm-hitting prob-complex} in the present paper, we formulate Theorems 
2.3 and 2.5 in \cite{Xiao2009} as follows.  
 
\begin{lemma} \label{Lem:dim} %\cite[Theorem 2.5]{Xiao2009}
Consider the interval $I$ of the form \eqref{def-interval}. Suppose that the assumptions (A1) , (A2) hold. Let $B \subseteq \bR^n$ 
be a Borel set such that $\dimh (B) \ge n - Q$. Then the following statements hold:
\begin{enumerate}
	\item[(a)] Almost surely,
	\begin{align*}
		\dimh (W^{-1}(B) \cap I)
		\le \min_{1 \le i \le N} \bigg\{ \sum_{j=1}^i \dfrac{H_i}{H_j} + N - i - H_i (n - \dim_H(B)) \bigg\}.
	\end{align*}
	\item[(b)] Assume that $\dimh (B) > n - Q$ and there is a finite constant $c_{7} \ge 1$ with the following property: 
For every $\eta \in (0, \, \dimh (B))$ there is a finite Borel measure 
	$\mu_{\eta}$ with compact support in $B$ such that
\begin{equation}\label{Eq:mu0}
\mu_{\eta}\big(\mathfrak B_\rho(x)\big) \le c_{7}\, \rho^{\eta} \qquad \hbox{ for all }\
x \in \bR^n  \hbox{ and } \rho > 0. 
\end{equation}
Then with positive probability,
	\begin{align*}
		\dimh (W^{-1}(B) \cap I)
		\ge \min_{1 \le i \le N} \bigg\{ \sum_{j=1}^i \dfrac{H_i}{H_j} + N - i - H_i (n - \dim_H(B)) \bigg\}. 
	\end{align*}
\end{enumerate}
\end{lemma}

The key feature of Condition (\ref{Eq:mu0}) is that the constant $c_{7}$ is independent of $\eta$, even though the 
probability measure $\mu_{\eta}$ may depend on $\eta$. For the proofs of Theorems \ref{Thm-hitting prob-real} 
and \ref{Thm-hitting prob-complex}, we take  $\mu_{\eta}$ as the restriction of the Lebesgue measure on the manifolds  
${\mathbf S}(d;k) \cap \mathfrak B_{\delta_0}(x_0)$ and ${\mathbf H}(d;k) \cap \mathfrak B_{\delta_0}(x_0)$, respectively. 
Both of these measures are independent  of $\eta$ and they satisfy (\ref{Eq:mu0}).

%\begin{lemma} \cite[Theorem 2.3]{Xiao2009}
%Consider the interval $I$ of the form \eqref{def-interval}. Suppose that the assumptions $(H1), (H2)$ hold. Let $G \subseteq \bR^n$ 
%be a Borel set such that $\dimh (G) \ge n - Q$. Then the following statement hold:
%\begin{enumerate}
%	\item[(I)] Almost surely,
%	\begin{align*}
%		\dimh (W^{-1}(G) \cap I)
%		\le \min_{1 \le i \le r} \left\{ \sum_{j=1}^i \dfrac{H_i}{H_j} + r - i - H_i (n - \dim_H(G)) \right\}
%	\end{align*}
%	\item[(II)] If $\dimh (G) > n - Q$, then for every $\epsilon > 0$,
%	\begin{align*}
%		\dimh (W^{-1}(G) \cap I)
%		\ge \min_{1 \le i \le r} \left\{ \sum_{j=1}^i \dfrac{H_i}{H_j} + r - i - H_i (n - \dim_H(G)) \right\} - \epsilon,
%	\end{align*}
%	on an event of positive probability (which may depend on $\epsilon$).
%\end{enumerate}
%Combining (I) and (II) yields:
%\[
%\big\|\dimh (W^{-1}(G) \cap I)\big\|_\infty =  \min_{1 \le i \le r} \left\{ \sum_{j=1}^i \dfrac{H_i}{H_j} + r - i - H_i (n - \dim_H(G)) \right\},
%\]
%where $\| \cdot\|_\infty$ is the essential supremum on the probability space $(\Omega, \cF, \bP)$.
%\end{lemma}

\bigskip

{\bf Acknowledgement. } The research of Y. Xiao is partially supported by NSF grant DMS-1855185.

\bibliographystyle{plain}
\bibliography{Collision_of_eigenvalues.bib}

\end{document}